\documentclass[preprint,review,10pt,authoryear,sort&compress]{elsarticle}

\usepackage{graphicx,fullpage}
\usepackage{amsmath,amssymb,dsfont}
\usepackage{color}
\usepackage[english]{babel}
\usepackage{mathrsfs}
\usepackage[colorlinks]{hyperref} 
\usepackage{amsthm}
\newtheorem{lemma}{Lemma}[section]
\newtheorem{corollary}[lemma]{Corollary}
\newtheorem{theorem}[lemma]{Theorem}
\theoremstyle{definition}
\newtheorem{definition}{Definition}
\newtheorem{example}{Example}
\theoremstyle{remark}
\newtheorem{remark}{Remark}[section]

\let\ALP \mathbf
\let\FLD \mathcal

\newcommand\IND{\mathbb{I}}
\newcommand{\ind}[1]{\delta_{\left\{#1\right\}}}

\newcommand{\beq}[1]{\begin{eqnarray} #1 \end{eqnarray}}
\newcommand{\beqq}[1]{\begin{eqnarray*} #1 \end{eqnarray*}}
\renewcommand{\Re}{\mathbb{R}}
\newcommand{\Na}{\mathbb{N}}

\newcommand{\sgn}{\text{sgn}}

\newcommand{\ex}[1]{\mathbb{E}\left[#1\right]}

\newcommand{\ws}{\overset{w^*}{\rightharpoonup}}

\newcommand{\infnorm}[1]{\|#1\|_{\infty}}

\newcommand{\lf}[1]{\lim_{#1\rightarrow \infty}}
\newcommand{\lif}[1]{\underset{#1\rightarrow \infty}{\lim\inf}\:}

\journal{Journal of Mathematical Economics}

\begin{document}

\begin{frontmatter}



\title{The Topology of Information on the Space of Probability Measures over Polish Spaces}

\author[MB]{Martin Barbie}
\author[AG]{Abhishek Gupta}
\address[MB]{University of Cologne, Center for Macroeconomic Research, WiSo-Hochhaus, 7. Stock, Zi. 735, Albertus-Magnus-Platz, K\"{o}ln, Germany. \\ Email: \texttt{barbie@wiso.uni-koeln.de}\\[0.2cm]}
\address[AG]{Department of Aerospace Engineering, University of Illinois at Urbana-Champaign, Illinois, USA.\\ Email: \texttt{gupta54@illinois.edu}}
\begin{abstract}
We study here the topology of information on the space of probability measures over Polish spaces that was defined in \citet{hellwig1996}. We show that under this topology, a convergent sequence of probability measures satisfying a conditional independence property converges to a measure that also satisfies the same conditional independence property. This also corrects the proof of a claim in \citet[Lemma 4]{hellwig1996}. Additionally, we determine sufficient conditions on the Polish spaces and the topology over measures spaces under which a convergent sequence of probability measures is also convergent in the topology of information.
\end{abstract}

\end{frontmatter}

\section{Introduction}
\label{sec:intro}
Consider Polish spaces $\ALP A$, $\ALP B$ and $\ALP C$, with generic elements in the spaces denoted by $a$, $b$ and $c$, respectively. Let $\wp(\cdot)$  denote the space of probability measures over the space $(\cdot)$. If $\mu\in\wp(\ALP A\times\ALP B)$, then $\mu(\cdot|a)$ denotes the conditional measure on the space $\ALP B$ given an element $a\in\ALP A$.

Suppose $\{\mu_n\}_{n\in\Na}\subset\wp(\ALP A\times\ALP B\times\ALP C)$ be a converging sequence of measures such that $\mu_n$ converges to $\mu_0$ in some topology as $n\rightarrow\infty$, and for all $n\in\Na$, the conditional measure on the space $\ALP B$ given elements $a$ and $c$ is independent of $c$, that is, $\mu_n(db|a,c) =\mu_n(db|a)$ for $a\in\ALP A,c\in\ALP C$ $\mu_n$-almost everywhere (see Subsection \ref{subsec:condI} for a formal definition of conditional independence). A natural question that arises is whether the limit $\mu_0$ also satisfy this conditional independence condition, that is, does $\mu_0(db|a,c) =\mu_0(db|a)$ $\mu_0$-almost every $a\in\ALP A,c\in\ALP C$  hold? As we will soon see (by an example) in Section \ref{sec:toi}, if we endow the space of measures with the usual weak-* topology, then this conditional independence condition may not be satisfied in the limit. Thus, we must endow the space of measures with a stronger topology such that  if we take a convergence sequence (or 
net) of measures in that topology, then the conditional independence condition is maintained in the limit.

To see why a stronger topology is essential in certain problems, consider an optimization, a team, or a game problem, in which the actions of decision makers depend on their information. If a sequence of measures induced by the strategies of a decision maker is taken, then in the weak-* limit, the actions of the decision maker may become independent of the information of the decision maker, or may become dependent on some other random variables that are not observed by the decision maker. This may be unacceptable in many circumstances\footnote{The action of a decision maker becoming independent of the information in the limit is not necessarily troublesome; the decision maker can decide not to use the information while making a decision.}, as it may violate information or causality constraints of the problem. Due to this issue, several authors studying game or optimization problems have assumed specific structures on sequences of measures or corresponding conditional measures in order to ensure that the 
causality or the information constraints are not violated by the limiting measure (see, for example, \cite{jordan1977, milgrom1985, engl1995, jackson2012}, among several others). 

The failure of preserving causality or information constraints in the limit under usual topologies on measure spaces led \citet{hellwig1996} to define the {\it topology of convergence in information}\footnote{We prefer to use ``topology of information'' instead of {\it topology of convergence in information} for brevity.} on measure spaces, which, generally speaking, is stronger than the usual weak-* topology on the measure spaces. Under this stronger topology, a convergent sequence of measures preserves the conditional independence property in the limit. The purpose of this paper is twofold: (i) to study the structure of this new topology over measure spaces, and (ii) to identify sufficient conditions when convergence of a sequence of measures under any of the other well-known topologies on measure spaces imply convergence in the topology of information.

\subsection{Previous Work}
One of the first papers to make assumptions on a sequence of measures over general Polish spaces in order to preserve informational constraints in the limit is \citet{jordan1977}. The author considered an infinite horizon one-person discrete-time optimization problem in which the state of the nature evolved as time progressed, and the decision maker, at any time step, observed the realizations of the state until that time step and actions taken until the previous time step. The author studied the continuity properties of the value functions as a function of the distribution of the states of nature. In order to maintain the information and causality constraint, that is, the action at a time step must be a function of the past actions, and realizations of the past and the current states of the world, the author assumed that the conditional measure of the future states given the past and the current states is a continuous function of the realizations of the past and the current states. 

The continuity assumption on the conditional measures is restrictive, as pointed out by \citet{hellwig1996}. This motivated \citet{hellwig1996} to define the topology of information on the measure spaces, under which a convergent sequence of measures maintain the informational constraints in the limit. Further, using this topology, he obtained the continuity properties of the value function as a function of the distribution of exogenous states of nature variables and proved the existence and continuity of optimal strategies in infinite horizon optimization problems. There is a mistake in the crucial steps in the proof of one of the main results, Lemma 4, in \citet{hellwig1996}, which we also address and correct in this paper.

Independently, \citet{milgrom1985} considered a game of incomplete information where the type spaces and action spaces of the decision makers, respectively, are Polish spaces and compact metric spaces. They assumed certain absolute continuity condition on the joint measures over the product space of the type spaces of the decision makers. This absolute continuity assumption was crucial in showing that the limit of a weak-* convergent sequence of distributional strategies\footnote{A distributional strategy of a decision maker is the joint measure over the action and type spaces of a decision maker induced by an equivalence class of behavioral strategies of the decision maker. For a precise definition and details, the reader is referred to \citet{milgrom1985}.} retain informational constraints in the limit. \citet{engl1995} studied the continuity properties of Nash equilibrium correspondence, as a function of the joint measures over the type spaces (also known as beliefs), in games with incomplete information. 
The author assumed that the type space of the decision makers is countable and the beliefs on the type space converge in the topology of setwise convergence. Similar setups have been studied in \citet{kajii1998} and \citet{jackson2012} later on.

\subsection{Outline of the Paper}
The paper is organized as follows. We discuss some preliminary results in Section \ref{sec:prelim}. In Section \ref{sec:toi}, we motivate the discussions on why the topology of information is important, and then define the topology of information on the space of measures over the product of two Polish spaces. Section \ref{sec:main} is the main section of this paper, where we prove that a convergent sequence of measures in the topology of information maintains the conditional independence property in the limit and discuss how this result fixes the mistake in the proof of \citet[Lemma 4]{hellwig1996}. We study some topological properties of the topology of information in Section \ref{sec:top} and present an example that applies the concept of topology of information to show existence of an optimal solution to an optimization problem described in that section. Thereafter, we study the relation between the topology of information and other well-known topologies like weak-* topology, topology of setwise 
convergence and convergence assumptions made in the literature in Section \ref{sec:relation}. In particular, we identify certain sufficient conditions for a sequence of measures under which convergence in either topology imply convergence of that sequence in the topology of information. Finally, we conclude our discussion in Section \ref{sec:conclusion}.

\subsection{Notation}
Throughout the paper, we use the following notation. Let $\ALP X$ be a set and $X\subset\ALP X$ be a subset. Then, $X^\complement$ denotes the complement of the set $X$. Now, let $\ALP X$ be a topological space. The vector space of all bounded continuous functions on the topological space $\ALP X$ endowed with supremum norm $\infnorm{\cdot}$ is denoted by $C_b(\ALP X)$, that is, $C_b(\ALP X):=\{f:\ALP X\rightarrow \Re: f \text{ is continuous and }\infnorm{f}<\infty\}$. For a metric space $\ALP X$, we let $U_b(\ALP X)$ denote the vector space of all bounded uniformly continuous functions with the supremum norm.

We use $\FLD B(\ALP X)$ and $\FLD P(\ALP X)$ to denote, respectively, the Borel $\sigma$-algebra and the power set of a topological space $\ALP X$. For two topological spaces $\ALP X_1$ and $\ALP X_2$, $\FLD B(\ALP X_1)\otimes\FLD B(\ALP X_2)$ denotes the Borel $\sigma$-algebra generated by the set of sets $\{X_1\times X_2:X_1\in\FLD B(\ALP X_1), X_2\in\FLD B(\ALP X_2)\}$. 

The space of probability measures over the measurable space $(\ALP X,\FLD B(\ALP X))$ is denoted by $\wp(\ALP X)$. We let $\ind{\cdot}$ denote the Dirac measure over a point $\{\cdot\}$. Henceforth, we use $\wp_w(\ALP X)$ to denote the space of probability measures over the space $\ALP X$ endowed with the weak-* topology, which is defined to be the weakest topology such that the map $\mu\mapsto \int_{\ALP X} f\:d\mu$ is a continuous map for every $f\in C_b(\ALP X)$. If $\{\mu_{\alpha}\}\subset\wp(\ALP X)$ is a convergent net of measures converging to $\mu_0$ in weak-* topology, then we denote it by $\mu_\alpha\ws\mu_0$. For a measure $\mu\in\wp(\ALP X\times\ALP Y)$, $\mu^{\ALP X}\in\wp(\ALP X)$ denotes the marginal of $\mu$ onto the space $\ALP X$.

A Polish space is defined as a separable topological space which is completely metrizable. It is well known that the space of measures over Polish spaces with weak-* topology is also a Polish space (see \citet[p. 505]{ali2006} or \citet[Theorem 8.9.4, p. 213]{bogachev2006b}). Thus, if $\ALP A$ and $\ALP B$ are Polish spaces, then $\ALP A\times\ALP B$, $\wp_w(\ALP A)$, $\wp_w(\wp_w(\ALP A))$, $\wp_w(\ALP A\times\wp_w(\ALP B))$ are all Polish spaces. 

\section{Preliminaries}\label{sec:prelim}
Before we discuss the topology of information, we recall disintegration theorem for measures \citet[Theorem 5.3.1, p. 121]{amb2008}. The main statement of the theorem is that if $\ALP A$ and $\ALP B$ are Polish spaces and $\mu\in\wp(\ALP A\times\ALP B)$, then there exists a conditional measure on the space $\ALP B$ given an element $a\in\ALP A$. Thus, one can disintegrate a joint measure into a product of a conditional measure and a marginal. We state the following lemma without proof, the proof of which relies on the disintegration theorem.
\begin{lemma}[\citet{amb2008}]\label{lem:glue}
Let $\ALP A,\ALP B$ and $\ALP C$ be Polish spaces. Consider $\mu_1\in\wp(\ALP A\times\ALP B)$ and $\mu_2\in\wp(\ALP B\times\ALP C)$ such that $\mu^{\ALP B}_1= \mu^{\ALP B}_2$. Then, there exists a measure $\mu_0\in\wp(\ALP A\times\ALP B\times\ALP C)$ such that
\beq{\label{eqn:glue}\mu_0^{\ALP A\times \ALP B} = \mu_1,\quad \mu_0^{\ALP B\times\ALP C} = \mu_2.}
Moreover, $\mu_0$ is unique if either there exists a Borel measurable function $h_1:\ALP B\rightarrow\ALP A$ such that $\mu_1(da,db) = \ind{h_1(b)}(da)\mu^{\ALP B}_1(db)$, that is, $\mu_1(A\times B)=\int_B \delta_{\lbrace h_1(b)\rbrace} (A) \mu_1^{\ALP B}(db)$ for any $A\in \FLD B(\ALP A)$ and $B\in \FLD B(\ALP B)$, or there exists a Borel measurable function $h_2:\ALP B\rightarrow\ALP C$ such that $\mu_2(db,dc) = \ind{h_2(b)}(dc)$ $\mu^{\ALP B}_2(db)$.
\end{lemma}
\begin{proof}
For proof, the reader is referred to \citet[Lemma 5.3.2, pp. 122]{amb2008}. The first part is also proved in \citet[Theorem 1.1.10, p. 7]{dudley1999}.
\end{proof}

We use the following result on the convergence of marginals of measures.

\begin{lemma}[Convergence of Marginals]\label{lem:marg}
Let $\ALP X$ and $\ALP Y$ be Polish spaces. Suppose $\{\nu_n\}_{n\in\Na}\subset\wp_w(\ALP X\times\ALP Y)$ is a converging sequence of measures that converges to $\nu_0$ as $n\rightarrow\infty$ in weak-* topology. Define $\zeta_n(B) := \nu^{\ALP Y}_n(B) = \nu_n(\ALP X\times B)$ for all Borel sets $B\subset\ALP Y$ and $n\in\Na\cup\{0\}$. Then, $\lf{n}\zeta_n\ws\zeta_0$, that is, $\nu^{\ALP Y}_n$ converges to $\nu^{\ALP Y}_0$ in the weak-* topology.
\end{lemma}
\begin{proof}
This is a simple consequence of Lemma 5.2.1 in \citet[p. 118]{amb2008}.
\end{proof}

We now formally define conditional independence property of a measure in the next subsection.
\subsection{Conditional Independence}\label{subsec:condI}
We recall here the formal definition of conditional independence. Let $\mu\in \wp(\ALP A\times\ALP B\times\ALP C)$ be a probability measure. Let $\mu(\cdot\vert a)$ denote the regular conditional distribution of $\mu$ on $\ALP B\times \ALP C$ given $a\in \ALP A$. Then, the distribution $\mu$ is said to be conditionally independent given a point $a\in \ALP A$ if
\begin{align}\label{eqn:condI1}
   \mu(B\times C\vert a)=\mu(B\vert a) \mu(C\vert a)
\end{align}
for all $B\in \mathcal{B}(\ALP B)$, $C\in \mathcal{B}(\ALP C)$ and $\mu$-almost every $a\in \ALP A$. If $\mu(\cdot\vert a,c)$ denotes the regular conditional distribution of $\mu$ on $\ALP B$ given $a\in \ALP A$ and $c\in \ALP C$. Then, conditional independence of $\mu$ given $a$ is equivalent to
\begin{align}\label{eqn:condI2}
   \mu(B\vert a,c)=\mu(B\vert a)
\end{align}
for all $B\in \mathcal{B}(\ALP B)$, $\mu$-almost every $c\in \ALP C$ and $a\in \ALP A$ (see Lemma 2.7 in \citet{jordan1977} or any other probability textbook)\footnote{Jordan (1977) states things in terms of conditional independence from $\sigma$-algebras generated by random variables. These properties translate straightforwardly into the properties fro regular conditional distributions}. We will use the characterization of conditional given by (\ref{eqn:condI2}) in this paper. For condition (\ref{eqn:condI2}), we use the short-hand notation $\mu(db\vert a,c)=\mu(db\vert a)$.
\section{The Topology of Information}\label{sec:toi}
In this section, we present the definition of the topology of information on the space of measures over a product of two Polish spaces. This topology has been defined in \citet{hellwig1996}. For this topology, we answer the following question: If $\{\mu_n\}_{n\in\Na}\subset\wp(\ALP A\times\ALP B\times\ALP C)$ is a sequence of measures that converges to $\mu_0$ in some topology and each $\mu_n$ satisfies (\ref{eqn:condI1}) (or equivalently (\ref{eqn:condI2})), does (\ref{eqn:condI1}) also hold for $\mu_0$?

First, an example is presented that demonstrates that the usual weak-* topology does not retain conditional independence property in the limit if we consider a weak-* convergent sequence of measures.
\subsection{Motivation}
We first take a look at the following example.
\begin{example}\label{exm:ex2}
Let $\ALP A = \Re$, $\ALP B = \{-1,1\}$, $\ALP C = \Re$. In this example, $c$ is a bijective function of $b$, whereas $a$ is a noise corrupted version of $c$. Define $h_n:\ALP B\rightarrow\ALP C$ as $h_n(b) = b\left( 1+\frac{1}{n} \right)$. Let us define $\{\mu_n\}_{n\in\Na\cup\{0\}}\subset\wp(\ALP A\times\ALP B\times\ALP C)$ as
\beqq{\mu_n(da|b) &=& \frac{1}{2}\ind{h_n(b)+1}(da)+\frac{1}{2}\ind{h_n(b)-1}(da)\qquad \mu_n(dc|b) =\ind{h_n(b)}(dc), \\
\mu_0(da|b) &=& \frac{1}{2}\ind{b+1}(da)+\frac{1}{2}\ind{b-1}(da)\qquad\qquad\quad \mu_0(dc|b) =\ind{b}(dc),\\
\mu_n^{\ALP B} &=& \mu_0^{\ALP B} =  \frac{1}{2}\ind{-1}+\frac{1}{2}\ind{1}.}
Since $h_n(b)\rightarrow b$ as $n\rightarrow\infty$, we conclude that $\mu_n\ws \mu_0$. For every $n\in\Na$, we have
\beqq{\mu_n(db|a) &=& \ind{\text{sgn}(a)}(db) = \left\{\begin{array}{ll}
\ind{-1}(db) &\text{if } a\in\{-\frac{1}{n},-2-\frac{1}{n}\}\\
\ind{1}(db) &\text{if } a\in\{\frac{1}{n},2+\frac{1}{n}\}\\\end{array}\right.,\\
\mu_n(dc|a) &=& \ind{(1+\frac{1}{n})\text{sgn}(a)}(dc).}
This implies, given $a$, conditional independence holds. On the other hand, $\mu_0(db|a,c) = \ind{c}(db)\neq \mu_0(db|a)$ for $a=0$ as $\mu_0(db\vert a)=\frac{1}{2}\delta_{\lbrace -1\rbrace}+\frac{1}{2}\delta_{\lbrace 1\rbrace}$. In other words, given $a$, $b$ and $c$ are completely determined if the three-tuple $(a,b,c)$ are distributed according to measure $\mu_n,\: n\in\Na$, but the same does not hold if the they are distributed according to the measure $\mu_0$. Thus, the conditional independence property is lost in the limit. {\hfill $\Box$}  
\end{example}


The conditional independence property is crucial to show the existence of optimal strategies in the problems of optimization under uncertainty, where the decision makers have informational or causality constraints, such as the ones considered in \citet{jordan1977, hellwig1996, milgrom1985} and others. In such problems, if we take a convergent sequence of measures, then in the limit, we have to avoid situations where the control actions (lying in the space $\ALP C$) of the decision makers become independent of their information (lying in the space $\ALP A$) or dependent on some other random variables (lying in the space $\ALP B$) that they do not observe.

In the next subsection, we define the topology of information on probability measures over a product of two Polish spaces. The definition of this topology over measure spaces comes from \citet[p. 448]{hellwig1996}.  

\subsection{Topology of Information on Probability Measure Space}
We let $\wp_I(\ALP A\times\ALP B)$ denote the space of probability measures over the space $\ALP A\times\ALP B$ endowed with the topology of information. This is defined as follows: Let $\ALP N\subset \wp_w(\ALP A\times\wp_w(\ALP B))$ be the set of measures\footnote{We reserve this notation for the rest of this paper.} that are induced by measurable functions, that is, for every $\nu\in\ALP N$, there exists a measurable function $h_{\nu}:\ALP A\rightarrow\wp(\ALP B)$ such that $\nu(da,d\zeta) = \ind{h_{\nu}(a)}(d\zeta)\nu^{\ALP A}(da)$. Assume that $\ALP N$ is endowed with the subspace topology of $\wp_w(\ALP A\times\wp_w(\ALP B))$. There is a bijection, say $\psi:\wp(\ALP A\times\ALP B)\rightarrow\ALP N$, between the space of measures $\wp(\ALP A\times \ALP B)$ and $\ALP N$ which can be seen as follows: If $\mu\in\wp_I(\ALP A\times\ALP B)$, then there exists a unique measure $\nu\in\ALP N$ such that $\nu(da,d\zeta) = \ind{\mu(\cdot|a)}(d\zeta)\mu^{\ALP A}(da)$. Conversely, if $\nu\in\ALP N$, then there 
exists a unique $\mu\in\wp_I(\ALP A\times\ALP B)$ defined by $\mu(A\times B) =  \int_A \int_{\wp(\ALP B)} \zeta(B)\nu(da,d\zeta)$. Also recall that by the definition of conditional measure, the function that maps $a\mapsto\mu(\cdot|a)$ is a Borel measurable function from $\ALP A$ to $\wp_w(\ALP B)$. We now define the topology of information.
\begin{definition}[Topology of Information (\citet{hellwig1996})]\label{def:topinfo}
The topology of information is defined to be the coarsest topology on $\wp(\ALP A\times\ALP B)$, denoted by $\wp_I(\ALP A\times\ALP B)$, that makes the function $\psi:\wp_I(\ALP A\times\ALP B)\rightarrow\ALP N$ continuous. Thus, if a set $U\subset\wp_I(\ALP A\times\ALP B)$ is open, then there exists an open set $V\subset\ALP N$ such that $U=\psi^{-1}(V)$.
\end{definition}

The topology of information is stronger than the weak-* topology on the space of measures over the space $\ALP A\times\ALP B$, and we prove this later in Corollary \ref{cor:toiweak}. Since $\psi$ is one-to-one mapping, the space $\wp_I(\ALP A\times\ALP B)$ is homeomorphic to $\ALP N\subset\wp_w(\ALP A\times\wp_w(\ALP B))$ with $\psi$ and $\psi^{-1}$ being the homeomorphism between the spaces. This fact is a consequence of the following result.

\begin{lemma}
$\psi^{-1}:\ALP N\rightarrow\wp_I(\ALP A\times\ALP B)$ is continuous.
\end{lemma}
\begin{proof}
Let $U\subset\wp_I(\ALP A\times\ALP B)$ be open. Then, $\psi(U)$ is open in $\ALP N$ by Definition \ref{def:topinfo}. Thus, $\psi^{-1}$ is continuous.
\end{proof}
The above result is also stated in \citet[p. 449]{hellwig1996}. In the next section, we show that the a convergent sequence of measures in the topology of information retains the conditional independence property in the limit. 

\section{Limit of Convergent Sequences in Topology of Information}\label{sec:main}
This section is devoted to prove the main result of this paper, that is, the topology of information retains the conditional independence condition in the limit. In order to show this, we first prove a few auxiliary results, that are used in the main result of this section, Theorem \ref{thm:main}. 

\subsection{Auxiliary Results}

We need the following definition to prove a few results later.

\begin{definition}\label{def:consistent}
Let $\ALP X$, $\ALP Y$ and $\ALP Z$ be Polish spaces and let $\mu_1\in\wp(\ALP X\times\ALP Y)$ and $\mu_2\in\wp(\ALP Y\times\ALP Z)$. We say that the measures $\mu_1$ and $\mu_2$ are consistent if and only if $\mu^{\ALP Y}_1 = \mu^{\ALP Y}_2$. {\hfill$\Box$}
\end{definition}

The next lemma discusses some properties of the function that glues two consistent probability measures in a specific manner.
\begin{lemma}\label{lem:aux1}
Let $\ALP X$, $\ALP Y$ and $\ALP Z$ be Polish spaces. Let $\ALP M\subset\wp_w(\ALP X\times\ALP Y)\times\wp_w(\ALP Y\times\ALP Z)$ be the set of all consistent measure pairs, and let $\tilde{\ALP M}\subset\ALP M$ be a tight set of consistent measure pairs. Define $\chi_1:\ALP M\rightarrow \wp_w(\ALP X\times\ALP Y\times\ALP Z)$ as
\beq{\label{eqn:chi1}\chi_1(\mu,\nu)(X\times Y\times Z) = \int_{Y} \mu(X|b)\nu(Z|b)\mu^{\ALP Y}(db) = \int_{Y\times Z} \mu(X|b)\nu(db,dc).}
for all Borel sets $X\subset\ALP X,Y\subset\ALP Y,Z\subset\ALP Z$. Then, 
\begin{enumerate}
\item $(\chi_1(\mu,\nu))^{\ALP X\times\ALP Y} = \mu $ and $(\chi_1(\mu,\nu))^{\ALP Y\times\ALP Z} = \nu$.
\item $\chi_1(\tilde{\ALP M})$ is a tight set of measures.
\item Let there exists a Borel measurable function $h_0:\ALP Y\rightarrow\ALP X$ such that $\mu_0(dx,dy) = \ind{h_0(y)}(dx)\mu^{\ALP Y}_0(dy)$. If $\{(\mu_n,\nu_n)\}_{n\in\Na}\subset\ALP M$ is a convergent sequence with the limit $(\mu_0,\nu_0)\in\ALP M$, then $\lf{n}\chi_1(\mu_{n},\nu_{n}) = \chi_1(\mu_{0},\nu_{0})$.
\end{enumerate}
\end{lemma}
\begin{proof}
See \ref{app:aux1}.
\end{proof}

This leads us to the following result, which is a corollary of the Lemma \ref{lem:aux1}.
\begin{corollary}\label{cor:cor1}
Let $\ALP M \subset \ALP N\times\wp_w(\ALP A\times\ALP C)$ be a set of consistent measure pairs. Then, $\varphi_1:\ALP M\rightarrow\wp_w(\ALP A\times\wp_w(\ALP B)\times\ALP C)$, defined in an identical fashion as $\chi_1$ in \eqref{eqn:chi1}, is a continuous function.
\end{corollary}
\begin{proof}
The proof follows from the third part of the result in Lemma \ref{lem:aux1}.
\end{proof}
The next lemma is also an important result. 
\begin{lemma}\label{lem:aux2}
Let $\ALP X$ and $\ALP Y$ be Polish spaces, and $\chi_2:\wp_w(\ALP X\times\wp_w(\ALP Y))\rightarrow\wp_w(\ALP X\times\ALP Y)$ be defined as
\beq{\label{eqn:chi2}\chi_2(\nu)(X\times Y) = \int_{X}\int_{\wp_w(\ALP Y)} \zeta(Y)\nu(dx,d\zeta)}
for all Borel sets $X\subset\ALP X,Y\subset\ALP Y$. Then, the following holds:
\begin{enumerate}
 \item For any bounded measurable function $g:\ALP X\times\ALP Y\rightarrow\Re$, define $\bar{g}:\ALP X\times\wp_w(\ALP Y)\rightarrow\Re$ as $\bar{g}(x,\zeta) = \int_{\ALP Y} g(x,y)\zeta(dy)$. We have
\beq{\label{eqn:bargxy}\int_{\ALP X\times\ALP Y}g(x,y)\chi_2(\nu)(dx,dy) = \int_{\ALP X\times\wp_w(\ALP Y)} \bar{g}(x,\zeta)\nu(dx,d\zeta).}
\item If $g$ is a bounded continuous function on its domain, then $\bar{g}$ is a bounded continuous function on its domain.
\item $\chi_2$ is a continuous function.
\end{enumerate}
\end{lemma}
\begin{proof}
See \ref{app:aux2}.
\end{proof}

\begin{corollary}\label{cor:cor2}
$\varphi_2:\wp_w(\ALP A\times\wp_w(\ALP B)\times\ALP C)\rightarrow\wp_w(\ALP A\times\ALP B\times\ALP C)$, defined in an identical fashion as $\chi_2$ in \eqref{eqn:chi2}, is a continuous function.
\end{corollary}
\begin{proof}
This is a direct application of Lemma \ref{lem:aux2}.
\end{proof}

We have now proved all the auxiliary results to prove the main result of this paper in the next subsection.

\subsection{Main Result}
We show that when we take a convergent sequence of measures in the topology of information, then conditional independence property holds. 
\begin{theorem}
\label{thm:main}
Let $\ALP P\subset \wp_I(\ALP A\times\ALP B)\times\wp_w(\ALP A\times\ALP C)$ be a set of consistent measure pairs. Let $\varphi:\ALP P\rightarrow\wp_w(\ALP A\times\ALP B\times\ALP C)$ be defined as 
\beqq{\varphi(\mu,\nu)(A\times B\times C) = \int_{A\times C} \mu(db|a)\nu(da,dc).}
Then, $\varphi$ is a continuous function on $\ALP P$ and $\varphi(\mu,\nu)(db|a,c) = \varphi(\mu,\nu)(db|a)=\mu(db|a)$. 
\end{theorem}
\begin{proof}
Note that $\varphi(\mu,\nu) = \varphi_2(\varphi_1(\psi(\mu),\nu))$, where $\varphi_1$ and $\varphi_2$ are defined in Corollaries \ref{cor:cor1} and \ref{cor:cor2}, respectively. It is clear that $\varphi$ is continuous since $\varphi_1,\varphi_2$ and $\psi$ are continuous functions on their domain. We now show that conditional independence property of $\varphi$. 

Let $g\in C_b(\ALP A\times\ALP B\times\ALP C)$ and $\bar{g}(a,\zeta,c):=\int_{\ALP B} g(a,b,c)\zeta(db)$. Then, $\bar{g}\in C_b(\ALP A\times\wp_w(\ALP B)\times\ALP C)$ by Lemma \ref{lem:aux2} Part 2, which further implies
\beqq{\int_{\ALP A\times\ALP B\times\ALP C}g(a,b,c)\varphi(\mu,\nu)(da,db,dc) &=& \int_{\ALP A\times\wp_w(\ALP B)\times\ALP C}\bar{g}(a,\zeta,c)\:\varphi_1(\psi(\mu),\nu)(da,d\zeta,dc),\\
&=& \int_{\ALP A\times\wp_w(\ALP B)\times\ALP C}\bar{g}(a,\zeta,c)\ind{\mu(\cdot|a)}(d\zeta)\nu(da,dc),\\
&=& \int_{\ALP A\times\ALP B\times\ALP C}g(a,b,c)\mu(db|a)\nu(da,dc).}
In the expressions above, the first equality follows from the definition of $\varphi_2$, the second equality follows from the definition of $\varphi_1(\psi(\cdot),\cdot)$ and the third equality follows from the definition of $\bar{g}$. On the other hand
\beqq{\int_{\ALP A\times\ALP B\times\ALP C}g(a,b,c)\varphi(\mu,\nu)(da,db,dc) &=& \int_{\ALP A\times\ALP B\times\ALP C}g(a,b,c)\varphi(\mu,\nu)(db|a,c)\nu(da,dc),}
which follows from the fact that $(\varphi(\mu,\nu))^{\ALP A\times\ALP C}=\nu$ by the definition of $\varphi_1$ and $\varphi_2$ in Corollaries \ref{cor:cor1} and \ref{cor:cor2}. Since the above equality holds for all continuous bounded functions on space $\ALP A\times\ALP B\times\ALP C$, we conclude that $\varphi(\mu,\nu)(db|a,c) = \varphi(\mu,\nu)(db|a)=\mu(db|a)$, and the proof of the theorem is complete.
\end{proof}

Now, if we take a sequence of measures $\{(\mu_n,\nu_n)\}_{n\in\Na}\subset \ALP P$ which converges to $(\mu_0,\nu_0)\in\wp_I(\ALP A\times\ALP B)\times\wp_w(\ALP A\times\ALP C)$, then $\varphi(\mu_0) = \lf{n}\varphi(\mu_n)$. Moreover, we also conclude that \[\varphi(\mu_0)(db|a,c) = \varphi(\mu_0)(db|a)=\mu_0(db|a).\] 
Thus, the conditional independence is retained. We also have the following corollary.
\begin{corollary}\label{cor:toiweak}
The topology of information is a stronger topology than the usual weak-* topology on the space of measures. Thus, the space of probability measure endowed with the topology of information is a Hausdorff space.
\end{corollary}
\begin{proof}
The statement follows from taking $\ALP C$ to be a one-point space in the result of Theorem \ref{thm:main}. Second statement follows immediately from the first statement of the corollary.
\end{proof}

\begin{corollary}\label{cor:main2}
Using the same notation as in Theorem \ref{thm:main}, consider $\mu\in\wp_I(\ALP A\times\ALP B)$ and $\{\nu_n\}_{n\in\Na}\subset\wp_w(\ALP A\times\ALP C)$ such that $\mu^{\ALP A} = \nu^{\ALP A}_n$ for all $n\in\Na$. If $\nu_n\ws\nu$ for some $\nu\in\wp_w(\ALP A\times\ALP C)$, then $\varphi(\mu,\nu_n)\ws\varphi(\mu,\nu)$.
\end{corollary}
\begin{proof}
First note that by Lemma \ref{lem:marg}, we know that $\nu^{\ALP A}= \mu^{\ALP A}$. Consequently, the tuple $(\mu,\nu)$ is a consistent pair of measures. Since $\wp_I(\ALP A\times\ALP B)$ is a Hausdorff space by Corollary \ref{cor:toiweak}, the statement follows. 
\end{proof}

Next example illustrates that if we use $\wp_w(\ALP A\times\ALP B)$ instead of $\wp_I(\ALP A\times\ALP B)$ in the statement of Theorem \ref{thm:main}, then the function $\varphi$ may not be continuous.

\begin{example}
Let $\ALP A=\left\lbrace 0,1,\frac{1}{2},\frac{1}{3},...\right\rbrace , \ALP B=\left\lbrace \underline{b},\overline{b}\right\rbrace, \ALP C=\left\lbrace \underline{c},\overline{c}\right\rbrace $, where $\ALP A$ is endowed with the subspace topology of the real line, and $\ALP B$ and $\ALP C$ is endowed with discrete topology. Consider two sequences $\{\mu_n\}_{n\in\Na} \subset \wp(\ALP A\times \ALP B)$ and $\{\nu_n\}_{n\in\Na}\subset \wp(\ALP A\times \ALP C)$ given by 
\beqq{\mu_n:=\frac{1}{2}\ind{(\frac{1}{n},\underline{b})}+\frac{1}{2}\ind{(0,\overline{b})}\qquad \nu_n = \frac{1}{2}\ind{(\frac{1}{n},\underline{c})}+\frac{1}{2}\ind{(0,\overline{c})}\qquad n\in\Na.}
Clearly, $\nu_n\ws\nu$ and $\mu_n\ws\mu$ in the weak-* topology, where $\mu:=\frac{1}{2}\ind{(0,\underline{b})}+\frac{1}{2}\ind{(0,\overline{b})}$ and $\nu := \frac{1}{2}\ind{(0,\underline{c})}+\frac{1}{2}\ind{(0,\overline{c})}$. Also, note that $\mu_n$ and $\nu_n$ have identical marginal distributions on $\ALP A$ for all $n\in\Na$. Let $f\in C_b(\ALP A\times \ALP B\times \ALP C)$ with $f(0,\overline{b},\underline{c})<f(0,\overline{b},\overline{c})$ and $f(0,\underline{b},\overline{c})
<f(0,\underline{b},\underline{c})$. Then we have
\beqq{\int f(a,b,c)d\varphi(\mu_n,\nu_n)=\int_{\ALP A\times \ALP C} \int_\ALP B f(a,b,c)\mu_n\left(db\vert a\right)\nu_n(da,dc) =\frac{1}{2} f\left(\frac{1}{n},\underline{b},\underline{c}\right)+\frac{1}{2} f(0,\overline{b},\overline{c})}
so that  this converges to $\frac{1}{2} f(0,\underline{b},\underline{c})+\frac{1}{2} f(0,\overline{b},\overline{c})$ as $n\rightarrow\infty$. On the other hand, we have
\beqq{ \int f(a,b,c)d\varphi(\mu,\nu) &=&\int_{\ALP A\times \ALP C} \int_\ALP B f(a,b,c)\mu\left(db\vert a\right) \nu(da,dc) \\
&=&\frac{1}{2} \left( \frac{1}{2}f(0,\underline{b},\underline{c})+\frac{1}{2}f(0,\overline{b},\underline{c})\right) +\frac{1}{2}\left( \frac{1}{2}f(0,\overline{b},\overline{c})+\frac{1}{2}f(0,\underline{b},\overline{c})\right).}
Given the assumptions on $f$, we have no equality. So, $\varphi(\nu_n,\mu_n)$ does not converge to $\varphi(\nu,\mu)$ as $n\rightarrow\infty$ in the weak-* topology.{\hfill$\Box$}
\end{example}

\subsection{Revisiting Example \ref{exm:ex2}}
Recall Example \ref{exm:ex2}. We show that the sequence of measures $\{\mu_n^{\ALP A\times\ALP B}\}_{n\in\Na}$ does not converge in the topology of information. We use here the same notation as in the example. First, note that
\beqq{\mu_n(db|a) = \ind{\sgn(a)}(db),\quad 
\mu_n^{\ALP A} = \sum_{i\in\{-1,1\}} \frac{1}{4} \Big(\ind{h_n(i)+1}+\ind{h_n(i)-1}\Big).}
We now show that the sequence $\{\mu_n^{\ALP A\times\ALP B}\}_{n\in\Na}$ does not converge to $\mu_0^{\ALP A\times\ALP B}$ in the topology of information. For any continuous function $f\in C_b(\ALP A\times\wp_w(\ALP B))$ such that $f(a,\ind{1})\neq f(a,\ind{-1})$, we get
\beqq{\int f(a,\zeta) \ind{\ind{\sgn(a)}}\mu_n^{\ALP A}(da)
= \sum_{i\in\{-1,1\}} \frac{1}{4} \Big(f(h_n(i)+1,\ind{h_n(i)+1}) +f(h_n(i)-1,\ind{h_n(i)-1})\Big).}
The right side of the equation above does not converge as $n\rightarrow\infty$, because $\sgn(h_n(1)-1) = \sgn(\frac{1}{n})$ and $\sgn(h_n(-1)+1) = -\sgn(\frac{1}{n})$ do not converge as $n\rightarrow\infty$. Furthermore, since $\ALP B = \{-1,1\}$, $f(a,\ind{0})$ is not well-defined. Using a similar approach as above, one can show that the sequence $\{\mu_n^{\ALP A\times\ALP C}\}_{n\in\Na}$ does not converge in the topology of information to $\mu_0^{\ALP A\times\ALP C}$. 

\begin{remark}
In the above setting, if either sequence $\{\mu_n^{\ALP A\times\ALP B}\}_{n\in\Na}$ or sequence $\{\mu_n^{\ALP A\times\ALP C}\}_{n\in\Na}$ converges in the topology of information, then the conditional independence property ($b$ and $c$ are conditionally independent given $a$) would hold. However, for the example we constructed, both sequences fail to converge in the topology of information, which implied that the conditional independence property failed to hold for the limit, $\mu_0$.{\hfill $\Box$}
\end{remark}

\subsection{Relation to \texorpdfstring{\citet{hellwig1996}}{TEXT}}
Our main result Theorem \ref{thm:main} can be used to give a correct proof of Lemma 4 in \citet{hellwig1996}. Hellwig considers a infinite-horizon sequential optimization problem. We present here a simpler version of the problem of Hellwig, and we refer the reader to \citet{hellwig1996} for a detailed description of the original model.

The basic structure of Hellwig's problem can be set up as follows: Let $\ALP A$ denote the set of exogenous states today, $\ALP B$ is the set of states tomorrow and $\ALP C$ is the set of actions chosen today. There is an exogenous probability distribution on the states of the world today and tomorrow given by $\nu\in \wp(\ALP A\times \ALP B)$. Any choice of the decision maker is represented by a probability measure on action and both the states, $\mu\in \wp(\ALP A\times \ALP B \times \ALP C)$. Note that $\mu$ is consistent with $\nu$ in the sense that $\mu^{\ALP A\times \ALP B}=\nu$. The choice of actions that the decision maker can take is restricted by  a correspondence $\beta(\nu)$, which assigns the possible choices of joint distributions of states and action for each exogenous measure $\nu$. Besides the technological constraints in $\beta$, it captures an informational constraint that the action chosen today cannot depend on the information (the state of the world) revealed tomorrow. Formally, this 
requires that conditioned on the state of the world today $a$, the action today $c$ and the state of the world 
tomorrow $b$ have to be conditionally independent. This means that for each $\mu\in \beta(\nu)$, we must have $\varphi(\nu,\mu^{\ALP A\times \ALP C})=\mu$, where $\varphi$ is defined in Theorem \ref{thm:main}. Lemma 4 in \citet{hellwig1996} aims to show that if the set of exogenous states is endowed with the topology of information, then the correspondence $\beta$ is upper-hemicontinuous. Since this topology is metrizable (see Lemma \ref{lem:met} below), the proof relies on the sequential characterization of upper-hemicontinuous correspondences.

If $\nu_n$ converges to $\nu$ in $\wp_I(\ALP A\times \ALP B)$, $\mu_n\in \beta(\nu_n)$, and $\mu_n\ws \mu$ to some $\mu\in\wp_w(\ALP A\times \ALP B\times \ALP C)$, it must be true that $\mu\in\beta(\nu)$. It is easy to verify the technological constraints imposed by $\beta$. However, in order to belong to $\beta(\nu)$,$\mu$ has to satisfy the conditional independence property. This follows now immediately from Theorem \ref{thm:main}, since $\mu_n=\varphi(\nu_n,\mu_n^{\ALP A\times \ALP C})\ws\varphi(\nu,\mu^{\ALP A\times \ALP C})$, so by the uniqueness of limits in the weak-* topology, $\mu=\varphi(\nu,\mu^{\ALP A\times \ALP C})$.

\citet{hellwig1996} instead tries to prove the preservation of conditional independence in the limit as follows: He considers $\psi(\nu)\in \ALP N\subseteq \wp_w(\ALP A\times \wp_w(\ALP B))$ and $\mu^{\ALP A\times \ALP C}$ on page 452 of his paper, applies the mapping $\varphi_1$ to these measures. He claims that (on p. 452) that 
\beqq{\varphi_1(\psi(\nu),\mu^{\ALP A\times \ALP C})(A\times B\times C) = \frac{\psi(\nu)(A\times B)\mu^{\ALP A\times \ALP C}(A\times C)}{\mu^{\ALP A}(A)}.}
for any $A\in \FLD B(\ALP A)$ with $\mu^{\ALP A}(A)>0$ and $B\in \FLD B(\wp_w(\ALP B))$ and $C\in \FLD B(\ALP C)$, and applies this equality to prove the conditional independence. However, the equality in the equation above is not true in general. We provide a counterexample to this claim now.

Consider $\ALP A = \{a_1,a_2\}, \:\ALP B =\{b_1,b_2\}$ and $\ALP C = \{c_1,c_2\}$. Now, define $h_1:\ALP A\rightarrow\ALP B$ and $h_2:\ALP A\rightarrow\ALP C$ as
\beqq{h_1(a_i) = b_i,\qquad h_2(a_i) = c_i,\quad i=1,2.}
Let $\nu\in\wp(\ALP A\times\ALP B)$ and $\mu\in\wp(\ALP A\times\ALP C)$ be probability measures, respectively, induced by functions $h_1$ and $h_2$, with the marginal on $\ALP A$ as the uniform distribution:
\beqq{\nu^{\ALP A}(da) = \mu^{\ALP A}(da) = \frac{1}{2}\ind{a_1}(da)+\frac{1}{2}\ind{a_2}(da).}
Consider $\psi(\nu)$, which assigns probabilities $\frac{1}{2}$ to $(a_1,\ind{b_1})$ and $(a_2,\ind{b_2})$. Now $\varphi_1(\psi(\nu),\mu)\in\wp_w(\ALP A\times\wp_w(\ALP B)\times\ALP C)$ is given by
\beqq{ \varphi_1(\psi(\nu),\mu)= \frac{1}{2} \ind{(a_1,\ind{b_1},c_1)}+\frac{1}{2} \ind{(a_2,\ind{b_2},c_2)}.}
Now, let $A = \ALP A,\: B = \lbrace\ind{b_1}\rbrace$ and $C = \{c_2\}$, we get
\beqq{\frac{\psi(\nu)(A\times B)\mu(A\times C)}{\nu^{\ALP A}(A)}=\frac{1}{2}\times\frac{1}{2} = \frac{1}{4},\\
\text{but}\qquad\varphi_1(\psi(\nu),\mu)(A\times B\times C) &=& 0,}
which is what we wanted to show. This completes the counterexample.
\section{Topological Properties of \texorpdfstring{$\wp_I(\ALP A\times\ALP B)$}{TEXT}}\label{sec:top}
In this section, we study a few topological properties of the space of measures endowed with topology of information. Since $\wp_I(\ALP A\times\ALP B)$ and $\ALP N$ are homeomorphic, we can show that $\wp_I(\ALP A\times\ALP B)$ is a metrizable separable space, which is also proved in \citet[Lemma 2, p. 449]{hellwig1996}.
\begin{lemma}\label{lem:met}
$\wp_I(\ALP A\times\ALP B)$ is a metrizable and separable space.   
\end{lemma}
\begin{proof}
Let $\rho_{(\cdot)}$ denote the metric on a metric space $(\cdot)$. For any $\mu_1,\mu_2\in \wp_I(\ALP A\times\ALP B)$, one can just the take metric 
\beqq{\rho_{\wp_I(\ALP A\times\ALP B)}(\mu_1,\mu_2) = \rho_{\ALP N}(\psi(\mu_1),\psi(\mu_2)).}
It is easy to verify that the above definition is indeed a metric on space $\wp_I(\ALP A\times\ALP B)$. Since $\ALP N$ is a subset of a separable space, we conclude that $\wp_I(\ALP A\times\ALP B)$ is also separable under this metric. This completes the proof of this lemma.
\end{proof}

However, it is easy to show that $\ALP N$ is not a closed subset of $\wp_w(\ALP A\times\wp_w(\ALP B))$\footnote{One can use a variation of Example \ref{exm:set} to show this fact.}. Thus, we cannot conclude that $\wp_I(\ALP A\times\ALP B)$ is complete under the metric defined in the proof above.


Corollary \ref{cor:toiweak} allows us to state the following result.

\begin{lemma}\label{lem:funcweak}
Let $\ALP T$ be a topological space. We use the same notation as in Theorem \ref{thm:main}.
\begin{enumerate}
\item Let $f:\wp_w(\ALP A\times\ALP B\times\ALP C)\rightarrow \ALP T$ be any function, and define $\bar{f}:\wp_I(\ALP A\times\ALP B)\times\wp_w(\ALP A\times\ALP C)\rightarrow \ALP T$ as $\bar{f}(\mu,\nu) = f(\varphi(\mu,\nu))$ for any $\mu\in\wp(\ALP A\times\ALP B)$ and $\nu\in\wp(\ALP A\times\ALP C)$. If $f$ is lower (resp. upper) semi-continuous, then $\bar{f}$ is lower (resp. upper) semi-continuous. Thus, if $f$ is continuous, then so is $\bar{f}$.
\item Let $c:\ALP A\times\ALP B\times\ALP C\rightarrow\Re\cup\{-\infty,\infty\}$ be a measurable function. Define $\bar{f}_c:\wp_I(\ALP A\times\ALP B)\times\wp_w(\ALP A\times\ALP C)\rightarrow\Re$ as $\bar{f}_c(\mu,\nu) = \int_{\ALP A\times\ALP B\times\ALP C} c\:d\varphi(\mu,\nu)$ for $\mu\in\wp(\ALP A\times\ALP B)$ and $\nu\in\wp(\ALP A\times\ALP C)$. If $c$ is lower semi-continuous and bounded from below, then $\bar{f}_c$ is a lower semi-continuous function. If $c$ is upper semi-continuous and bounded from above, then  $\bar{f}_c$ is an upper semi-continuous function. Thus, if $c$ is continuous and bounded, then so is $\bar{f}_c$.
\end{enumerate}
\end{lemma}
\begin{proof}
The proof of Part 1 follows from Theorem \ref{thm:main} and Corollary \ref{cor:toiweak}. Part 2 follows immediately from Part 1 of the result along with Lemma 4.3 of \citet[p. 43]{villani2009}.  
\end{proof}
We now present an example below that applies the result of Lemma \ref{lem:funcweak} to prove the existence of an optimal solution to an optimization problem. The proof technique adopted in the following example can be extended to a game of incomplete information or any sequential optimization with perfect recall.
\begin{example}\label{exm:opt}
Let $\ALP B$ be the state space of the world, $\ALP A$ be the observation space of a decision maker and $\ALP C$ be the decision space of the decision maker. Assume that $\ALP C$ is a compact space and $c:\ALP A\times\ALP B\times\ALP C\rightarrow[0,\infty)$ is a continuous function. Further, assume that observation $a$ and state $b$ are correlated with each other and let $\mu\in\wp(\ALP A\times\ALP B)$ denote the joint probability measure over the observation space and the state space. Let $\Gamma$ denote the space of all measurable functions $\gamma:\ALP A\rightarrow\ALP C$. The question now is that under what conditions, there exists a measurable function $\gamma^\star\in\Gamma$ such that the following holds
\beq{\label{eqn:cabc} \ex{c(a,b,\gamma^\star(a))} = \inf_{\gamma\in\Gamma} \ex{c(a,b,\gamma(a))} :=  \inf_{\gamma\in\Gamma}\int_{\ALP A\times\ALP B} c(a,b,\gamma(a))\mu(da,db).}
We now show that the aforementioned optimization problem admits an optimal solution, thereby showing the existence of an optimal $\gamma^\star$. We show the existence result in four steps.


{\it Step 1:} Let us first expand the strategy space of the decision maker to include all randomized strategies as well. Thus, each decision maker decides on a conditional measure $\nu(dc|a)$ such that the measure on $\ALP A\times\ALP C$ is given by $\nu(da,dc):=\nu(dc|a)\mu^{\ALP A}(da)$, and let $\ALP P\subset\wp_w(\ALP A\times\ALP C)$ denote the set of all such $\nu$. It is immediate that $\ALP P$ is a tight and weak-* closed set of measures, thus weak-* compact. Since the space $\ALP P$ subsumes the measures induced by strategies in $\Gamma$, we have
\beq{\label{eqn:infNG}\inf_{\tilde \nu\in\ALP P} \int c(a,b,c)\tilde \nu(dc|a)\mu(da,db)\leq\inf_{\gamma\in\Gamma}\int_{\ALP A\times\ALP B} c(a,b,\gamma(a))\mu(da,db).}

{\it Step 2:} Let us endow the space of measures over $\ALP A\times\ALP B$ with the topology of information. Consider the sequence $\{\mu_n\}_{n\in\Na}$, defined by $\mu_n=\mu$ for all $n\in\Na$. Since $\wp_I(\ALP A\times\ALP B)$ is a metric space, the sequence $\{\mu_n\}_{n\in\Na}$ converges to $\mu$ in the topology of information. 


{\it Step 3:} Now, consider a sequence $\{\nu_n\}_{n\in\Na}\subset\ALP P$ satisfying
\beqq{\int c(a,b,c)\nu_n(dc|a)\mu(da,db)< \inf_{\tilde \nu\in\ALP P} \int c(a,b,c)\tilde \nu(dc|a)\mu(da,db)+\frac{1}{n}.}
Since $\ALP P$ is weak-* compact, there exists a convergent subsequence, say $\{\nu_{n_k}\}_{k\in\Na}\subset \{\nu_n\}_{n\in\Na}$, such that it converges to some $\nu^\star\in\ALP P$. A consequence of this result is that $\varphi(\mu_{n_k},\nu_{n_k})\ws \varphi(\mu,\nu^\star)$ as $k\rightarrow\infty$. Since $c$ is continuous and bounded from below, applying the result of Lemma \ref{lem:funcweak} Part 2, we conclude that 
\beqq{\int c(a,b,c)\nu^\star(dc|a)\mu(da,db)\leq  \lif{n}\int c(a,b,c) \nu_n(dc|a)\mu(da,db),}
which further implies
\beqq{\int c(a,b,c)\nu^\star(dc|a)\mu(da,db)= \inf_{\tilde \nu\in\ALP P} \int c(a,b,c)\tilde \nu(dc|a)\mu(da,db).}
Hence, we know that there exists an optimal randomized strategy of the decision maker.

{\it Step 4:} Now, we can apply {\it Blackwell's principle of irrelevant information} (see \citet{blackwell1964}, \citet[p. 457]{yukselbook} for details) to conclude that there exists a measurable function, say $\gamma^\star:\ALP A\rightarrow\ALP C$, such that 
\beqq{\int c(a,b,\gamma^\star(a))\mu(da,db)=\int c(a,b,c)\nu^\star(dc|a)\mu(da,db),}
which, together with \eqref{eqn:infNG}, completes the proof of existence of an optimal solution to the optimization problem posed in \eqref{eqn:cabc}.
{\hfill$\Box$}
\end{example}

\begin{remark}
Instead of formulating the optimization problem over the space $\wp_I(\ALP A\times\ALP B)\times\wp_w(\ALP A\times\ALP C)$ in the example above, if we had formulated it over the space $\wp_w(\ALP A\times\ALP B\times\ALP C)$, then we could show that there exists a $\lambda^\star\in \wp_w(\ALP A\times\ALP B\times\ALP C)$ such that
\beqq{\int c\: d\lambda^\star= \inf_{\tilde \lambda\in\wp_w(\ALP A\times\ALP B\times\ALP C)} \int c\:d\tilde \lambda}
using similar arguments as above, but to show the conditional independence property (state and action are independent given the observation) of the limiting measure, we will have to use a similar approach as used in Lemmas \ref{lem:aux1} and \ref{lem:aux2}. Thus, formulating the optimization problem over a product space $\wp_I(\ALP A\times\ALP B)\times\wp_w(\ALP A\times\ALP C)$ that uses topology of information makes it easier to show the conditional independence property  of the limiting decision strategy.{\hfill$\Box$}
\end{remark}

The example stated above also shows how to apply topology of information to solve optimization problems. A similar approach, with certain modifications, can be used to analyze game problems. In the next section, we study the relation between other well-known topologies over measure spaces and the topology of information.

\section{Relation to Other Topologies on Measure Spaces}\label{sec:relation}
In this section, we show that under some conditions, convergence of a sequence of measures in some well-known topologies on the space of measures - weak-* topology, topology of setwise convergence, and the norm topology (the topology induced by total variation norm), implies convergence of that sequence in the topology of information. First, we state definitions of setwise convergence and convergence in total variation of a sequence of measures for a Borel space $(\ALP X,\FLD B(\ALP X))$.
\begin{definition}[Setwise Convergence of measures]
A sequence of measures $\{\nu_n\}_{n\in\Na}$ over the space $\ALP X$ is said to converge setwise to a measure $\nu_0$ if
\beqq{\lf{n}\nu_n(X) = \nu_0(X)}
for every measurable set $X\subset\ALP X$.{\hfill$\Box$}
\end{definition}
\begin{definition}[Convergence of measures in total variation norm]
A sequence of measures $\{\nu_n\}_{n\in\Na}$ over the space $\ALP X$ is said to converge in total variation to a measure $\nu_0$ if
\beqq{\lf{n}\|\nu_n-\nu_0\|_{TV} = 0,}
where the total variation norm of any countably additive signed measure $\nu$ is defined to be
\beqq{\|\nu\|_{TV} := \sup_{f:\ALP X\rightarrow[-1,1]} \int_{\ALP X} f\:d\nu,}
where the supremum is taken over all functions $f$ that are Borel measurable.{\hfill$\Box$}
\end{definition}

In the next few subsections, we identify certain sufficient conditions on the sequences of measures, such that if the sequence converges under some topology, then it implies that the sequence also converges in the topology of information. 
 
\subsection{Relation to the Weak-* Topology}
Example \ref{exm:ex2} is an example of the case where weak-* convergence of a sequence of measures does not imply convergence of that sequence in the topology of information. Thus, it is clear that the notion of weak-* convergence is not sufficient to guarantee conditional independence property of limiting measures. However, under some restrictive assumptions, weak-* convergence implies convergence in the topology of information. Our next two results identify two such sets of conditions.

\begin{theorem}\label{thm:wstoi}
Let $\ALP X$ be a Polish space and $\ALP Y$ be a locally compact Polish space. Consider a sequence $\lbrace \mu_n\rbrace_{n\in\Na}$ of probability measures on $\ALP X\times \ALP Y$ such that each $\mu_n$ has a (measurable) density $f_n$ with respect to $\mu^{\ALP X}_n\otimes\mu^{\ALP Y}_n$. Further, assume that (i) $\{\mu_n\}_{n\in\Na}$ converges to $\mu$ in the weak-* topology, (ii) $\mu$ has a continuous density $f$ with respect to $\mu^{\ALP X}\otimes\mu^{\ALP Y}$, and (iii) $f_n$ converges uniformly to $f$ on each compact subset of $\ALP X\times \ALP Y$ as $n\rightarrow\infty$. Then, $\lbrace \mu_n\rbrace_{n\in\Na}$ converges to $\mu$ as $n\rightarrow\infty$ in $\wp_I(\ALP X\times\ALP Y)$. 
\end{theorem}
\begin{proof}
See \ref{app:wstoi}.
\end{proof}
   
The conditions above is motivated from the convergence condition in \citet{milgrom1985}, which considers convergence results for Bayesian games\footnote{In fact, they require that the limit density be a.s.-continuous with respect to the product measure. They also do not require the spaces to be locally compact. We impose a slightly stronger assumption in order to obtain a general result.}. We have the following corollary of the theorem above.

\begin{corollary}\label{cor:wstoi2}
Let $\ALP X,\ALP Y_1, \ALP Y_2$ and $\ALP Z$ be locally compact Polish spaces. Consider weak-* sequences of measures $\{\mu_n\}_{n\in\Na}\subset\wp_w(\ALP X\times\ALP Y_1\times\ALP Y_2)$ and $\{\nu_n\}_{n\in\Na}\subset\wp_w(\ALP Y_1\times\ALP Z)$ such that $\mu_n^{\ALP Y_1} = \nu_n^{\ALP Y_1}$, $\mu_n\ws \mu$ and $\nu_n\ws\nu$ for some $\mu\in\wp_w(\ALP X\times\ALP Y_1\times\ALP Y_2)$ and $\nu\in\wp_w(\ALP Y_1\times\ALP Z)$.  Assume that (i) $\mu_n$ has a (measurable) density $f_n$ with respect to $\mu^{\ALP X}_n\otimes\mu^{\ALP Y_1}_n\otimes\mu^{\ALP Y_2}_n$ for every $n\in\Na$, (ii) $\mu$ has a continuous density $f$ with respect to $\mu^{\ALP X}\otimes\mu^{\ALP Y_1}\otimes\mu^{\ALP Y_2}$, and (iii) $f_n$ converges uniformly to $f$ on each compact subset of $\ALP X\times \ALP Y_1\times\ALP Y_2$ as $n\rightarrow\infty$. Define a sequence of measures $\{\lambda_n\}_{n\in\Na}\subset\wp(\ALP X\times\ALP Y_1\times\ALP Y_2\times\ALP Z)$ and $\lambda\in\wp(\ALP X\times\ALP Y_1\times\ALP Y_2\times\ALP Z)$ as 
\beq{\label{eqn:lambdan}\lambda_n(dx,dy_1,dy_2,dz) = \nu_n(dz|y_1)\mu_n(dx,dy_1,dy_2)\;\; n\in\Na,\quad \lambda(dx,dy_1,dy_2,dz) = \nu(dz|y_1)\mu(dx,dy_1,dy_2).}
Then, the following holds:
\begin{enumerate}
\item The sequence $\lbrace \mu_n\rbrace_{n\in\Na}$ converges to $\mu$ as $n\rightarrow\infty$ in three topological spaces: $\wp_I(\ALP X\times(\ALP Y_1\times\ALP Y_2))$ and $\wp_I(\ALP Y_i\times(\ALP X\times\ALP Y_j))$, where $i,j\in\{1,2\}, i\neq j$.
\item The sequence $\{\lambda_n\}_{n\in\Na}$ converges to $\lambda$ as $n\rightarrow\infty$ in the weak-* topology.
\item The sequence $\{\lambda_n\}_{n\in\Na}$ converges to $\lambda$ as $n\rightarrow\infty$ in the space $\wp_I(\ALP Y_2\times(\ALP X\times\ALP Y_1\times\ALP Z))$.  
\end{enumerate}
\end{corollary}
\begin{proof}
Theorem \ref{thm:wstoi} implies the first part of the corollary. The second part then follows from Theorem \ref{thm:main}. We now prove Part 3 of the corollary. 

Let $\tilde{\ALP Y}_1 = \ALP Y_1\times\ALP Z$. Also note that, by the definition of $\lambda_n$ and $\lambda$ in \eqref{eqn:lambdan}, $\lambda^{\tilde{\ALP Y}_1}_n = \nu_n$ and $\lambda^{\tilde{\ALP Y}_1} = \nu$. Define $\tilde f_n,\tilde f:\ALP X\times\ALP Y_2\times\tilde{\ALP Y}_1\rightarrow\Re$ as $\tilde f_n(x,y_2,(y_1,z)) = f_n(x,y_1,y_2)$ for $n\in\Na$ and $\tilde f(x,y_2,(y_1,z)) = f(x,y_1,y_2)$. Then, for any $n\in\Na$, we have $\mu_n(dx,dy_1,dy_2) = f(x,y_1,y_2) \lambda^{\ALP X}_n(dx)\lambda^{\ALP Y_1}_n(dy_1)\lambda^{\ALP Y_2}_n(dy_2) $, which further implies
\beqq{\lambda_n(dx,dy_1,d_2,dz) &=&  \nu_n(dz|y_1)\mu_n(dx,dy_1,dy_2) = f_n(x,y_1,y_2) \lambda^{\ALP X}_n(dx)\lambda^{\ALP Y_2}_n(dy_2) \Big(\nu_n(dz|y_1)\lambda^{\ALP Y_1}_n(dy_1)\Big),\\
 & =& \tilde f_n(x,y_2,(y_1,z)) \lambda^{\ALP X}_n(dx)\lambda^{\ALP Y_2}_n(dy_2)\lambda^{\tilde{\ALP Y}_1}_n(d y_1,dz).}
A similar result holds for $\lambda$. Thus, the following statements follow immediately from the definitions above, hypotheses of the corollary, and the above equations:
\begin{enumerate}
\item[(i)] $\lambda_n$ is absolutely continuous with respect to the measure $\lambda^{\ALP X}_n\otimes\lambda^{\ALP Y_2}_n\otimes\lambda^{\tilde{\ALP Y}_1}_n$ for every $n\in\Na$ with the Radon-Nikodym derivative as $\tilde f_n$.
\item[(ii)] $\lambda$ has a continuous density $\tilde f$ with respect to $\lambda^{\ALP X}\otimes\lambda^{\ALP Y_2}\otimes\lambda^{\tilde{\ALP Y}_1}$. \item[(iii)] $\tilde f_n$ converges uniformly to $\tilde f$ on each compact subset of $\ALP X\times \ALP Y_2\times\tilde{\ALP Y}_1$ as $n\rightarrow\infty$.
\end{enumerate}
The above statements, Part 2 of the corollary, together with the result of Theorem \ref{thm:wstoi}, imply that the sequence $\{\lambda_n\}_{n\in\Na}$ converges to $\lambda$ as $n\rightarrow\infty$ in the space $\wp_I(\ALP Y_2\times(\ALP X\times\tilde{\ALP Y}_1))=\wp_I(\ALP Y_2\times(\ALP X\times\ALP Y_1\times\ALP Z))$.   This completes the proof of the corollary.
\end{proof}

The above corollary is useful in optimization or game problems in which multiple decision makers act simultaneously based on their observations. To see this, consider a game or an optimization problem with $N\in\Na$ decision makers. Let $\ALP X$, $\ALP Y_i$ and $\ALP Z_i$ denote the state space of the nature, observation space of decision maker $i$ and the decision space of decision maker $i$, respectively, for $i\in\{1,\ldots,N\}$. Assume that all the spaces are locally compact Polish spaces, and the joint distribution of the state and the observations of the decision makers admits a continuous density function with respect to the product measure of their marginals. Then, the result of Corollary \ref{cor:wstoi2} can be used iteratively to conclude that a weak-* convergent sequence of joint measures over state, observation and action spaces of the decision makers, induced by appropriate strategies of the decision makers, maintains conditional independence properties\footnote{In this setup, the number of 
conditional independence properties to check are the same as the number of decision makers.} in the limit.    

Our proof of Theorem \ref{thm:wstoi} relies on a property of measure spaces over a locally compact Polish space with weak-* topology. Therefore, it is not clear as of now if the restriction of locally compact Polish spaces can be weakened in the hypotheses of Theorem \ref{thm:wstoi} and its corollary above.

A somewhat different condition was considered in \cite{jordan1977}, in which the conditional measure on $\ALP Y$ given $x$ is assumed to be continuous in $x$. In our next theorem, we show that under such an assumption with another condition, weak-* convergence of a sequence of measures imply convergence in the topology of information.  

\begin{theorem}\label{thm:wstoi3}
Let $\ALP X$ and $\ALP Y$ be Polish spaces. Let $\{\mu_n\}_{n\in\Na} \subset \wp_w(\ALP X\times\ALP Y)$ be a weak-* convergent sequence of measures, converging to $\mu$ as $n\rightarrow\infty$. For each $n\in\Na$, let $f_n:\ALP X\rightarrow\wp_w(\ALP Y)$ be the measurable function defined as $f_n(x)(\cdot) = \mu_n(\cdot|x)$. Similarly, define measurable function $f:\ALP X\rightarrow\wp_w(\ALP Y)$ as $f(x)(\cdot) = \mu(\cdot|x)$. Let $E\subset\ALP X$ be the set of all $x\in\ALP X$ such that there exists a sequence $\{x_n\}_{n\in\Na}\subset\ALP X$ such that $x_n\rightarrow x$, but the sequence $\{f_n(x_n)\}_{n\in\Na}$ does not converge to $f(x)$. If $\mu^{\ALP X}(E) = 0$, then $\lbrace \mu_n\rbrace_{n\in\Na}$ converges to $\mu$ as $n\rightarrow\infty$ in $\wp_I(\ALP X\times\ALP Y)$. 
\end{theorem}
\begin{proof}
Define $h_n:\ALP X\rightarrow\ALP X\times\wp_w(\ALP Y)$ to be the function as $h_n(x) = (x,f_n(x))$, and similarly define $h(x) = (x,f(x))$. Then, for any $x\in E^\complement$, if $\{x_n\}_{n\in\Na}\subset\ALP X$ is a sequence converging to $x$, then the sequence $\{h_n(x_n)\}_{n\in\Na}$ converges to $h(x)$. Since $\mu^{\ALP X}_n\ws\mu^{\ALP X}$, we know from \citet[Theorem 5.5, p. 34]{billing1968} or \citet[Theorem 8.4.1 (iii), p. 195]{bogachev2006b} that for any continuous function $g\in C_b(\ALP X\times\wp_w(\ALP Y))$, we have
\beqq{\lf{n}\int_{\ALP X\times\wp_w(\ALP Y)} g(h_n(x))d\mu^{\ALP X}_n = \int_{\ALP X\times\wp_w(\ALP Y)}g(h(x))d\mu^{\ALP X}.}
This implies that $\mu_n$ converges to $\mu_0$ as $n\rightarrow\infty$ in $\wp_I(\ALP X\times\ALP Y)$, which completes the proof of the theorem.
\end{proof}

\citet{jackson2012} also make similar assumptions as mentioned in Theorems \ref{thm:wstoi} and \ref{thm:wstoi3} above. For one part, their assumption on page 207 of ther paper is similar to the one in \citet{milgrom1985}, which we discuss in Theorem \ref{thm:wstoi}. For other parts, on page 206 of their paper, they assume a sort of uniform continuity version of the original continuity assumption of \citet{jordan1977}. Theorem \ref{thm:wstoi3} can be used to show that both of these assumptions imply convergence in the topology of information. We refer the reader to \ref{app:jordan} to see how the result of Theorem \ref{thm:wstoi3} can be applied to the setting considered in \citet{jordan1977}. 

To see how the assumptions of page 206 in \citet{jackson2012} are a special case of Theorem \ref{thm:wstoi3}, recall that their assumption requires that $\mu_n$ converges weakly to $\mu_\infty$. Further, if $d_{\ALP X}$ denotes some fixed metric generating the topology on $\ALP X$, they require that for each $\varepsilon>0$ and each continuous function $f:\ALP X\times \ALP Y\rightarrow \left[ 0,1\right]$, there exists $N\in \mathbb{N}$ and $\delta>0$ such that for all $m,n>N$ (including $n=\infty$) and for all $x,x'\in \ALP X$ with $d_{\ALP X}(x,x')<\delta$
\begin{align}\label{eqn:Jackson}
     \left\vert \int_{\ALP Y} f(x,y)d\mu_n(y\vert x)-\int_{\ALP Y} f(x',y)d\mu_m(y\vert x') \right\vert<\varepsilon.
\end{align}
Now, let $g:\ALP Y\rightarrow \left[ 0,1\right]$ be a continuous function. Fix some $x_\infty\in \ALP X$ and consider a sequence $\{x_m\}_{m\in\Na}$ such that $x_m\rightarrow x_\infty$ as $m\rightarrow\infty$. In \eqref{eqn:Jackson} above, take $f=g$, $n=\infty$, $x=x_\infty$ and $x'=x_m$. Then, (\ref{eqn:Jackson}) implies that $\mu_m(\cdot\vert x_m)$ converges to $\mu_\infty(\cdot\vert x_\infty)$ as $m\rightarrow\infty$ in weak-* topology\footnote{Weak-* convergence for conditional distributions in Theorem \ref{thm:wstoi3} requires one to consider convergence for all bounded continuous functions, not just those mapping to $\left[ 0,1\right]$. However, by adding a constant and rescaling, it suffices to show convergence for functions mapping to $\left[ 0,1\right]$.}. From Theorem \ref{thm:wstoi3}, it follows that $\mu_n$ converges to $\mu_\infty$ as $n\rightarrow\infty$ in $\wp_I(\ALP X\times\ALP Y)$.

\subsection{Relation to the Topology of Setwise Convergence}
In this subsection, we state sufficient conditions when setwise convergence of measures imply convergence in topology of information. Before we state the conditions, let us first consider an example where setwise convergence of a sequence of measures {\it does not} imply convergence in the topology of information.

\begin{example}\label{exm:set}
This example uses Rademacher functions, which were used in \citet[p. 445]{hellwig1996} to construct an example for discontinuous behavior of conditional distributions under weak convergence. Let $\left[0,1\right[$ with Borel $\sigma$-algebra $\FLD B_{\left[ 0,1\right[}$ be given. Let $\lambda$ denote the Lebesgue measure restricted to $\left[ 0,1 \right[$. Recall that the $n$-th Rademacher function is defined as $F_n(\omega)=\sum_{k=0}^{2^{n-1}-1} \mathbb{I}_{\left[ \frac{2k}{2^n},\frac{2k+1}{2^n}\right[ }(\omega)$ for any $\omega\in \left[ 0,1\right[$. We have the following result.

\begin{lemma}\label{lem:cl1}
Define a sequence of measures $\{\lambda_n\}_{n\in\Na}$ such that $\lambda_n(A) := \lambda\left( A\cap F_n^{-1}(1)\right)$ for any $A\in\FLD B_{\left[ 0,1\right[}$ and $n\in\Na$. For any $A\in \FLD B_{\left[ 0,1\right[}$,
\begin{align}\label{rademacher}
\lambda_n(A)=\lambda\left( A\cap F_n^{-1}(1)\right) \longrightarrow \frac{1}{2}\lambda(A) \quad\text{ as }\quad n\rightarrow\infty.
\end{align}
In other words, $\lambda_n\rightarrow\frac{1}{2}\lambda$ as $n\rightarrow\infty$ in the setwise topology over the space of measures over $[0,1[$.
\end{lemma}
\begin{proof}
See \ref{app:cl1}.
\end{proof}
Note that then also $\lambda\left( A\cap F_n^{-1}(0)\right) \longrightarrow \frac{1}{2}\lambda(A)$ as $n\rightarrow\infty$. Let $\ALP X := [0,1[$ and $\ALP Y:=\{1,2\}$. We consider now the set $\ALP X\times\ALP Y$ with $\sigma$-algebra $\FLD B(\ALP X)\otimes\mathcal{P}(\lbrace 1,2\rbrace)$. Let $\mu$ be the measure given by
\begin{align*}
   \mu(A)=\int_{\left[ 0,1\right[} \mu(A_x\vert x) d\lambda
\end{align*}
where $A_x$ is the $x-$section of the set $A\in \FLD B(\ALP X)\otimes\FLD P(\ALP Y)$ and for each $x\in \left[ 0,1\right[$ we set $\mu(\lbrace 1\rbrace\vert x)=\mu(\lbrace 2\rbrace\vert x)= \frac{1}{2}$. For each $n$, we set $\mu_n(\lbrace 1\rbrace\vert x)=1$ if $F_n(x)=0$ and $\mu_n(\lbrace 1\rbrace\vert x)=0$ if $F_n(x)=1$. The measure $\mu_n\in\wp(\ALP X\times\ALP Y)$ is defined for each $A\in \FLD B(\ALP X)\otimes\FLD P(\ALP Y)$ as
\beqq{\mu_n(A)=\int_{\left[ 0,1\right[} \mu_n(A_x\vert x) d\lambda}

For each $A\in \FLD B(\ALP X)\otimes\FLD P(\ALP Y)$, let
$A_1:=\left\lbrace x\in \ALP X \vert A_x=\lbrace 1\rbrace  \right\rbrace$, $A_2:=\left\lbrace x\in \ALP X \vert A_x=\lbrace 2\rbrace  \right\rbrace$ and $A_{12}:=\left\lbrace x\in \ALP X \vert A_x=\lbrace 1,2\rbrace  \right\rbrace$. Note that $A_1,A_2$ and $A_{12}$ are all measurable sets. We have
\beqq{\mu_n(A)&=&\int_{A_1} (1-F_n) d\lambda+\int_{A_2} F_n d\lambda+\int_{A_{12}} \mu_n(\lbrace 1,2\rbrace\vert x) d\lambda\\
&=& \lambda(A_1\cap F_n^{-1}(0))+\lambda(A_2\cap F_n^{-1}(1))+\lambda(A_{12}),\\
\mu(A)&=&\int_{A_1} \mu(\lbrace 1\rbrace\vert x) d\lambda+\int_{A_2} \mu(\lbrace 2\rbrace\vert x) d\lambda+\lambda(A_{12})=\frac{1}{2}\lambda(A_1)+\frac{1}{2}\lambda(A_2)+\lambda(A_{12}).}
Lemma \ref{lem:cl1} implies that $\mu_n(A)\rightarrow \mu(A)$ for each $A\in \FLD B(\ALP X)\otimes\FLD P(\ALP Y)$, that is, $\mu_n$ converges to $\mu$ setwise. On the other hand, $\mu_n$ does not converge to $\mu$ in the topology of information, as $\mu(\cdot|x)=\frac{1}{2}\ind{1}+\frac{1}{2}\ind{2}$ for all $x\in\ALP X$ and $\mu_n(\cdot|x)$ is either equal to $\ind{1}$ or $\ind{2}$.{\hfill$\Box$}
\end{example}

It turns out that if $\ALP X$ is countable with discrete metric and $\ALP Y$ is a Polish space, then setwise convergence of a sequence of measures over $\ALP X\times\ALP Y$ implies convergence of that sequence of measures in the topology of information. This assumption corresponds to the convergence assumption in \citet{engl1995}.
\begin{theorem}\label{thm:settoi}
Let $\ALP X$ be a countable space with discrete metric and $\ALP Y$ be a Polish space. If a sequence $\{\mu_n\}_{n\in\Na}$ of probability measures on $\ALP X\times \ALP Y$ converges setwise to a probability measure $\mu$, then the sequence $\lbrace \mu_n\rbrace_{n\in\Na}$ converges to $\mu$ as $n\rightarrow\infty$ in $\wp_I(\ALP X\times\ALP Y)$. 
\end{theorem}
\begin{proof}
See \ref{app:settoi}.
\end{proof}
We also have an immediate corollary to the above theorem.
\begin{corollary}\label{cor:settoi}
Under the same assumptions of Theorem \ref{thm:settoi}, the result of Theorem \ref{thm:settoi} holds if $\mu_n\rightarrow\mu$ as $n\rightarrow\infty$ in the metric induced by the total variation norm.
\end{corollary}
\begin{proof}
For a sequence of measures, convergence in total variation norm implies setwise convergence (see the discussion on page 291 of \cite{bogachev2006a}). This fact implies the result of the corollary.
\end{proof}

We can use Theorem \ref{thm:settoi} to conclude the following result for measures over countable discrete spaces.

\begin{theorem}
Let $\ALP X$ and $\ALP Y$ be countable spaces, each of which is endowed with the discrete metric. Let $\{\mu_n\}_{n\in\Na}\subset\wp_w(\ALP X\times\ALP Y)$ be a weak-* convergent sequence of measures, converging to $\mu$. Then, the sequence $\{\mu_n\}_{n\in\Na}$ converges to $\mu$ in $\wp_I(\ALP X\times\ALP Y)$.  
\end{theorem}
\begin{proof}
Note that both $\ALP X$ and $\ALP Y$ are Polish spaces. Furthermore, the space of measurable functions and the space of continuous functions over discrete countable spaces are the same. Since $\mu_n\ws\mu$, we know from that $\mu_n$ converges to $\mu$ in total variation norm, which further implies that $\mu_n$ converges to $\mu$ setwise, as $n\rightarrow\infty$. The result then follows from Theorem \ref{thm:settoi}.
\end{proof}
The assumption of countable spaces in the above theorem is part of the conditions on a convergent sequence of measures considered in \cite{kajii1998}. Besides this assumption, they require additional assumptions motivated by game theoretic considerations, which makes their convergence concept stronger than weak-* convergence (or in this case equivalently setwise or norm convergence) on a discrete countable space. Their assumptions on the convergent sequence of measures imply that the sequence converges in the topology of information.

\subsection{Discussion}

The theorems we proved in this section show that topology of information is weaker that other well-known topologies under certain conditions on the Polish spaces or underlying distributions of random variables. However, for general Polish spaces, it is not clear as of now if convergence of a sequence of measures in total variation norm implies convergence of that sequence in the topology of information or vice-versa. It could be possible that the topology of information is stronger than the topology of convergence in total variation metric under certain conditions, but we have been unable to prove this or construct a counterexample. Thus, we leave the following question as a topic for further research: 

{\bf Open Problem 1:} Let $\ALP X$ and $\ALP Y$ be Polish spaces, and let $\wp_{TV}(\ALP X\times\ALP Y)$ denote the space of probability measures over $\ALP X\times\ALP Y$ with the topology induced by total variation norm. What is the relation between the topological spaces $\wp_I(\ALP X\times\ALP Y)$ and $\wp_{TV}(\ALP X\times\ALP Y)$? Does there exists a sequence $\{\mu_n\}_{n\in\Na}\subset\wp(\ALP X\times\ALP Y)$ that converges to $\mu$ in the metric induced by the total variation norm, but does not converge to $\mu$ in the topology of information?{\hfill $\Box$}

We showed that a sequence of measures converging in the topology of information preserves conditional independence property in the limit. A natural question to ask would be if this is also necessary, that is, if a sequence of measures converging under some topology preserves conditional independence property in the limit, then does it also converge in the topology of information? In other words, the following problem is also interesting in its own right:

{\bf Open Problem 2:} Let $\ALP X$, $\ALP Y$, $\ALP Z$ be Polish spaces and let $\{\mu_n\}_{n\in\Na}\subset\wp(\ALP X\times\ALP Y\times\ALP Z)$ be a sequence of measures. Assume that (i) $\mu_n\rightarrow\mu$ for some measure $\mu$ as $n\rightarrow\infty$ in some topology, (ii) $\mu_n(dy,dz|x) = \mu_n(dy|x)\mu_n(dz|x)$ for $\mu_n$ almost every $x$, and (iii) $\mu(dy,dz|x) = \mu(dy|x)\mu(dz|x)$ for $\mu$ almost every $x$. Does either of the following holds: (i) $\{\mu_n^{\ALP X\times\ALP Y}\}_{n\in\Na}$ converges to $\mu^{\ALP X\times\ALP Y}$ in the space $\wp_I(\ALP X\times\ALP Y)$, or (ii) $\{\mu_n^{\ALP X\times\ALP Z}\}_{n\in\Na}$ converges to $\mu^{\ALP X\times\ALP Z}$ in the space $\wp_I(\ALP X\times\ALP Z)$?{\hfill$\Box$}

\section{Conclusion}\label{sec:conclusion}
In this paper, we studied the topology of information on the space of measures over Polish spaces. We showed, through examples, that the weak-* topology and the topology of setwise convergence are weaker notions of topology than the topology of information for probability measures over general Polish spaces. We also determined conditions under which weak-* convergence or setwise convergence of a sequence of measures implied convergence in the topology of information.

This topology is useful in game or optimization problems that feature informational constraints among the decision makers or causality constraints in the decision making process. For applications, we refer the reader to the papers listed in the introduction. In particular, \citet{hellwig1996} uses topology of information explicitly in showing the existence of optimal solution in an infinite horizon decision problem, while \cite{milgrom1985, jordan1977, kajii1998} assume specific conditions on the underlying spaces, sequences of measures, and/or the topology over measure spaces, to ensure that the limits satisfy information and causality constraints in the problem. 

Our results show that many notions of convergence used in this literature are a special case of convergence in the topology of information and we believe that the results of this paper will be helpful in unifying and extending conditions for existence and continuity of optimal solution or Nash equilibrium strategies in such problems. Recently, \citet{yuksel2012} considered one-person optimization of observation channels in dynamic decision problems. The technical difficulty in \citet{yuksel2012} arose partly due to the fact that the actions of a decision maker affected future states of the world, but the decision maker does not recall the past observations. It will be interesting to consider such dynamic decision making problems (for example, Markov decision problems) with or without memory, and use topology of information to identify conditions that guarantee existence of optimal decision rules.

\appendix

\section{Proof of Lemma \ref{lem:aux1}}\label{app:aux1}
Part 1 of the lemma follows directly from the definition. We prove Part 2 and 3 of the lemma here.
\begin{enumerate}
\item[2.] Let $\epsilon>0$. Since $\tilde{\ALP M}$ is a tight set of measure pairs, without loss of generality, we can assume that there exist compact sets $K_1\subset\ALP X, K_2\subset\ALP Y, K_3\subset\ALP Z$ such that
\beqq{\mu\left((K_1\times K_2)^\complement\right)<\epsilon,\quad\nu\left((K_2\times K_3)^\complement\right)<\epsilon.}
By Part 1 of the lemma, we know that $\mu(K_1\times K_2) = \chi_1(\mu,\nu)\left(K_1\times K_2\times \ALP Z\right)$. Moreover, note that since $\mu(K_1|y)\leq 1$ for all $b\in\ALP Y$, we get
\beqq{\chi_1(\mu,\nu)\left(K_1\times K_2\times K_3^\complement\right) = \int_{K_2\times K_3^\complement} \mu(K_1|y)\nu(dy,dc)\leq \nu\left((K_2\times K_3)^\complement\right)<\epsilon.}
Thus, using the two equations above, we get
\beqq{\chi_1(\mu,\nu)\left(K_1\times K_2\times K_3\right) = \mu(K_1\times K_2) - \chi_1(\mu,\nu)\left(K_1\times K_2\times K_3^\complement\right)> 1-2\epsilon,}
which proves Part 2 of the lemma. 

\item[3.]  Assume, contrary to the claim, that the sequence $\{\chi_1(\mu_{n},\nu_{n})\}_{n\in\Na}$ does not converge to $\chi_1(\mu_{0},\nu_{0})$ as $n\rightarrow\infty$. Since $\{\chi_1(\mu_{n},\nu_{n})\}_{n\in\Na}$ does not converge to $\chi_1(\mu_{0},\nu_{0})$, there exists an $\varepsilon_0>0$ and a subsequence $\{\chi_1(\mu_{n_l},\nu_{n_l})\}_{l\in\Na}$ such that
\beq{\label{eqn:rhoxyz}\rho_{\wp_w(\ALP X\times\ALP Y\times\ALP Z)}(\chi_1(\mu_{n_l},\nu_{n_l}),\chi_1(\mu_{0},\nu_{0}))>\varepsilon_0,}
where $\rho_{\wp_w(\ALP X\times\ALP Y\times\ALP Z)}(\cdot,\cdot)$ is the distance between the two measures under Prohorov's metric over the space $\wp_w(\ALP X\times\ALP Y\times\ALP Z)$. By Part 2 of the lemma, we can conclude that $\{\chi_1(\mu_{n},\nu_{n})\}_{n\in\Na}$ is a tight set of measures, which implies that for the subsequence $\{\chi_1(\mu_{n_l},\nu_{n_l})\}_{l\in\Na}$, there exists a further subsequence $\{(\mu_{n_{l_k}},\nu_{n_{l_k}})\}_{k\in\Na}$ such that $\chi_1(\mu_{n_{l_k}},\nu_{n_{l_k}})$ converges to some measure $\lambda_0\in\wp(\ALP X\times\ALP Y\times\ALP X)$. 

We now show that $\lambda_0 = \chi_1(\mu_{0},\nu_{0})$. First, note that 
\beqq{\lambda_0^{\ALP X\times\ALP Y} = \lf{k}(\chi_1(\mu_{n_{l_k}},\nu_{n_{l_k}}))^{\ALP X\times\ALP Y} = \lf{k}\mu_{n_{l_k}} = \mu_0,}
which follows from Lemma \ref{lem:marg}. Similarly, $\lambda_0^{\ALP Y\times\ALP Z} = \nu_0$. Now, since $\mu_0$ is induced by a Borel measurable function, applying Lemma \ref{lem:glue}, we conclude that $\lambda_0$ is the unique probability measure that glues the two probability measures $\mu_0$ and $\nu_0$, which implies $\lambda_0 = \chi_1(\mu_{0},\nu_{0})$. However, this contradicts \eqref{eqn:rhoxyz}. 

Thus, our assumption that the sequence $\{\chi_1(\mu_{n},\nu_{n})\}_{n\in\Na}$ does not converge to $\chi_1(\mu_{0},\nu_{0})$ as $n\rightarrow\infty$ is false. This completes the proof of the lemma.
\end{enumerate}

\section{Proof of Lemma \ref{lem:aux2}}\label{app:aux2}
\begin{enumerate}
\item Let $\Pi_{\ALP X}$ be the projection map. For all $D\in \mathcal{B}(\ALP X\times \ALP Y)$, note that 
\begin{align}\label{eqn:chi}
\chi_2(\nu)(D)=\int_{\Pi_{\ALP X}(D)}\int_{\wp_w(Y)} \zeta(D_x) \nu(d\zeta,dx),
\end{align}
is the unique measure defined by \eqref{eqn:bargxy}, where $D_x$ is the $\ALP Y$-section of the set $D$ at point $x$. First, we prove that for any $D\in\FLD B(\ALP X\times\ALP Y)$, the map $h_D:\wp_w(\ALP Y)\times\ALP X\rightarrow\Re$, defined by $h_D(\zeta,x) = \zeta(D_x)$, is $\mathcal{B}(\ALP X \times \wp_w(\ALP Y))$-measurable. To see this, define 
$\FLD D\subset\mathcal{B}(\ALP X\times\ALP Y)$ as follows:  
\begin{align}
\mathcal{D}=\lbrace D\in \mathcal{B}(\ALP X\times\ALP Y) \text{ such that }\hspace{1mm} \zeta(D_x) \hspace{1mm}\text{ is } \hspace{1mm}\mathcal{B}(\ALP X \times \wp_w(\ALP Y))-\text{measurable } \rbrace
\end{align}
Since $\ALP X$ and $\ALP Y$ are Polish spaces, the space $\ALP X$ resp. $\ALP Y$ have countable bases of open sets $\mathcal{O}$ resp. $\mathcal{U}$. We now show that $\mathcal{D}$ is a Dynkin-system that contains $\FLD C:=\{O\times U: O\in \mathcal{O}, U\in\mathcal{U}\}$. For a set $O\times U\in \FLD C$, we have to show that the following map 
\beqq{h(\zeta,x) = \left\{\begin{array}{cl} \zeta(U) & \text{if }x\in O\\ 0 & \text{otherwise} \end{array}\right.}
is $\mathcal{B}(\ALP X \times \wp_w(\ALP Y))$-measurable. This follows since $\lbrace (x,\zeta)\vert h(\zeta,x)\geqslant \alpha\rbrace$ is equal to $\ALP X\times \wp_w(\ALP Y)$ for $\alpha=0$ and equal to $O\times \lbrace \zeta\in\wp_w(\ALP Y) \text{ s.t. } \zeta(U)\geqslant \alpha\rbrace$ for $\alpha>0$. Since $U$ is an open set in $\ALP Y$, by Portmonteau theorem \citet{billing1968}, the map $\zeta\mapsto\zeta(U)$, as a function from $\wp_w(\ALP Y)$ to $\Re$, is upper-semicontinuous, and therefore measurable. So $\lbrace \zeta\in\wp_w(\ALP Y) \text{ s.t. } \zeta(U)\geqslant \alpha\rbrace \in \FLD B(\wp_w(\ALP Y))$. This shows the measurability of $h$. Since $\FLD C$ is closed under finite intersection and generates $\mathcal{B}(\ALP X\times\ALP Y)$, by Dynkin's lemma (see Theorem 2.4 in \citet{bauer2001}), $\mathcal{D}=\mathcal{B}(\ALP X\times \ALP Y)$. This implies that $h_D$ is measurable for any $D\in\FLD B(\ALP X\times\ALP Y)$, which further implies that \eqref{eqn:chi} is well defined. 

Further, \eqref{eqn:chi} is easily seen to define a probability measure, and for $D\in \FLD C$ \eqref{eqn:chi} reduces to \eqref{eqn:chi2}. By Theorem 5.4 in \citet{bauer2001}, the measure defined by \eqref{eqn:chi} is the unique extension of \eqref{eqn:chi2}. 

To see that \eqref{eqn:bargxy} holds, note that $\bar{g}(x,\zeta)$ is for $g(x,y)=\IND_D$ for $D\in \mathcal{B}(\ALP X\times \ALP Y)$ equal to $\zeta(D_x)$, so that \eqref{eqn:chi} is just \eqref{eqn:bargxy} for indicator functions. By standard arguments, \eqref{eqn:bargxy} holds for all simple functions, and therefore, holds for all bounded measurable functions from $\ALP X\times\ALP Y$ to $\Re$.

\item Let $\lbrace x_n,\zeta_n\rbrace_{n\in\Na}$ converge to $(x,\zeta)$. Since $\zeta_n\ws\zeta$, Prohorov's theorem implies that for every $\varepsilon>0$, there exists a compact set $K_{\varepsilon}\subseteq \ALP Y$ such that $\zeta_n(K_{\varepsilon})\geqslant 1-\varepsilon$ for all $n\in\Na$. Now we have
\begin{align*}
  &\left\vert \int_{\ALP Y}g(x_n,y) \zeta_n(dy)-\int_{\ALP Y}g(x,y) \zeta(dy)\right\vert \\
  &\leq \left\vert \int_{\ALP Y}g(x_n,y) \zeta_n(dy)-\int_{\ALP Y}g(x,y) \zeta_n(dy)\right\vert + \left\vert \int_{\ALP Y}g(x,y) \zeta_n(dy)-\int_{\ALP Y}g(x,y) \zeta(dy)\right\vert
\end{align*}
The second integral converges to zero as $n\rightarrow\infty$. For the first integral, we have
\begin{align*}
   \left\vert \int_{\ALP Y}g(x_n,y) \zeta_n(dy)-\int_{\ALP Y}g(x,y) \zeta_n(dy)\right\vert \leq \int_{K_{\varepsilon}}\vert g(x_n,y)-g(x,y)\vert \zeta_n(dy) +\int_{K^\complement_{\varepsilon}} \vert g(x_n,y)-g(x,y)\vert \zeta_n(dy)
\end{align*}
We now show that both integrands in the right side of the equation above converge to $0$ as $n\rightarrow\infty$. On the compact set $\Big(\{x_n\}_{n\in\Na}\bigcup \{ x\} \Big)\times K_{\varepsilon}$, the function $g$ is uniformly continuous\footnote{See Lemma 9.5 in \citet{stokey1989} for a similar argument}. Thus, for any $\kappa>0$, there exists a natural number $N_1\in\Na$ such that $\vert g(x_n,y)-g(x,y)\vert<\kappa$ for all $n\geq N_1$ and $y\in K_{\varepsilon}$. This implies that the first integrand converges to zero as $n\rightarrow\infty$. Also, $\vert g(x_n,y)-g(x,y)\vert$ is bounded as $g$ is bounded. Further, $\zeta_n(K^\complement_{\varepsilon}) < \varepsilon$ for all $n\in\Na$ implies that the second integrand can be made arbitrarily small (as $\varepsilon$ can be chosen arbitrary) for sufficiently large $n$. Thus, the second integrand converges to zero as $n\rightarrow\infty$. Boundedness is immediate given the boundedness of $g$. This proves the result.

\item Let $\{\nu_n\}_{n\in\Na}\subset\wp_w(\ALP X\times\wp_w(\ALP Y))$ be a weak-* convergent sequence with the limit $\nu_0$. Define $\mu_n:=\chi_2(\nu_n)$ for $n\in\Na\cup\{0\}$. We want to show that for any $g\in C_b(\ALP X\times\ALP Y)$,
\beqq{\lim_{n\rightarrow\infty}\int_{\ALP X\times\ALP Y} g(x,y)\mu_n(dx,dy) = \int_{\ALP X\times\ALP Y} g(x,y)\mu_0(dx,dy), }
which would imply that $\chi_2$ is continuous. In order to prove this, let us define $\bar{g}$ as $\bar{g}(x,\zeta) = \int_{\ALP Y} g(x,y)\zeta(dy)$. Part 2 of the lemma implies $\bar{g}\in C_b(\ALP X\times\wp(\ALP Y))$, which further implies
\beqq{\int_{\ALP X\times\ALP Y} g(x,y)\mu_n(dx,dy) = \int_{\ALP X\times\wp_w(\ALP Y)} \bar{g}\:d\nu_n\underset{n\rightarrow\infty}{\longrightarrow} \int_{\ALP X\times\wp_w(\ALP Y)} \bar{g}\:d\nu_0 = \int_{\ALP X\times\ALP Y} g(x,y)\mu_0(dx,dy),}
where equalities hold due to Part 1 of the lemma. This completes the proof of third part of the lemma.
\end{enumerate}

\section{Proof of Theorem \ref{thm:wstoi} }\label{app:wstoi}
To prove this, recall that convergence in the topology of information means that $\psi(\mu_n)$ converges to $\psi(\mu)$ as $n\rightarrow\infty$ in the weak-* topology. Let $g$ be a bounded and continuous function from $\ALP X\times \wp_w(\ALP Y)$ to the real numbers. Thus, we want to show that as $n\rightarrow\infty$
\beq{\label{eqn:convI} \int_{\ALP X\times \wp_w(\ALP Y)} g(x,\nu) d\psi(\mu_n)\longrightarrow  \int_{\ALP X\times \wp_w(\ALP Y)} g(x,\nu) d\psi(\mu).}
We now prove the \eqref{eqn:convI} in three steps. 

{\it Step 1:} We claim that  
\beq{\label{eqn:condq} \mu(B\vert x)=\int_B f(x,y) d\mu^{\ALP Y}\quad  \text{ for all } B \in \FLD B(\ALP Y)\quad \mu^{\ALP X}\text{-almost surely}.}
Let  $B\in\FLD B(\ALP Y)$. For any $A\in\FLD B(\ALP X)$, we have
\beqq{\int_{A}\mu(B|x)\mu^{\ALP X}(dx) = \mu(A\times B) = \int_{A}\left( \int_{B} f(x,y)\mu^{\ALP Y}(dy) \right)\mu^{\ALP X}(dx),}
which implies \eqref{eqn:condq}. A similar statement is true for $\mu_n$ and $f_n$ for any $n\in\Na$. For every $x\in\ALP X$ and $n\in\Na$, define $\nu^x_n\in\wp(\ALP Y)$ and $\nu^x\in\wp(\ALP Y)$ as
\beqq{\nu^x_n(\cdot):=\int_{(\cdot)} f_n(x,y)d\mu^{\ALP Y}_n,\qquad \nu^x(\cdot) = \int_{(\cdot)} f(x,y)d\mu^{\ALP Y}.}
Then, $\nu^x_n$ and $\nu^x$, respectively, are precisely the conditional measures $\mu_n(\cdot|x)$ and $\mu(\cdot|x)$ at $x\in\ALP X$ for $n\in\Na$.

{\it Step 2:} For any sequence $\lbrace x_n\rbrace_{n\in\Na}$ converging to $x$, we now show that the sequence of probability measures $\{\nu_n^{x_n}\}_{n\in\Na}$ converges to the probability measure $\nu^x$ as $n\rightarrow\infty$ in the weak-* topology. Note that on a locally compact space (as $\ALP Y$), weak-* convergence of a sequence of probability measures holds if and only if vague convergence of the sequence of measures holds \citet[Corollary 30.9, p. 197]{bauer2001}\footnote{See \citet[Chapter 30]{bauer2001} for more details on vague convergence and its relation to weak-* convergence of a sequence of measures over a locally compact Polish space}. Thus, we now prove that $\{\nu_n^{x_n}\}_{n\in\Na}$ converges to $\nu^x$ as $n\rightarrow\infty$ in vague topology.

Let $q:\ALP Y\rightarrow\mathbb{R}$ be bounded continuous function with a compact support denoted by $K_q\subset\ALP Y$. The sequence $\{\nu_n^{x_n}\}_{n\in\Na}$ converges vaguely to $\nu^x$ if and only if
\beqq{\int_\ALP Y q(y)f_n(x_n,y)d\mu^{\ALP Y}_n\longrightarrow \int_\ALP Y q(y)f(x,y)d\mu^{\ALP Y}\qquad \text{ as }n\rightarrow\infty.}
In order to prove that the equation above holds, note that
\beq{\left\vert \int_\ALP Y q(y)f_n(x_n,y)d\mu^{\ALP Y}_n-\int_\ALP Y q(y)f(x,y)d\mu^{\ALP Y}\right\vert & \leq & \left\vert \int_\ALP Y q(y)f_n(x_n,y)d\mu^{\ALP Y}_n-\int_\ALP Y q(y)f(x,y)d\mu^{\ALP Y}_n\right\vert\nonumber\\
 & & \hspace{3mm}+\left\vert \int_\ALP Y q(y)f(x,y)d\mu^{\ALP Y}_n-\int_\ALP Y q(y)f(x,y)d\mu^{\ALP Y}\right\vert.\label{eqn:qfnf}}
The last summand in \eqref{eqn:qfnf} converges to zero as $n\rightarrow\infty$, because for a fixed $x\in\ALP X$, $q(y)f(x,y)$ is bounded continuous function on the compact set  $K_q$ and $\mu^{\ALP Y}_n\ws\mu^{\ALP Y}$. We now show that the first summand on the right side of \eqref{eqn:qfnf} converges to zero as $n\rightarrow\infty$. By triangle inequality, we have
\begin{align*}
   \left\vert f_n(x_n,y)-f(x,y)\right\vert\leqslant \left\vert f_n(x_n,y)-f(x_n,y)\right\vert+\left\vert f(x_n,y)-f(x,y)\right\vert
\end{align*}
By assumption, $f_n$ converges uniformly to $f$ on compact sets. Also, since $f$ is continuous function, $f$ is uniformly continuous on the compact set $\left(\{x_n\}_{n\in\Na}\cup\{x\}\right)\times K_q$\footnote{See Lemma 9.5 in \citet{stokey1989} for a similar argument.}. Using these two facts, for any given $\varepsilon>0$ we can find a $N\in\Na$ such that for any $n>N$ and for all $y\in K_q$, we have
\beqq{\left\vert f_n(x_n,y)-f(x_n,y)\right\vert<\frac{\varepsilon}{2} \hspace{4mm}\text{and}\hspace{4mm} \left\vert f(x_n,y)-f(x,y)\right\vert<\frac{\varepsilon}{2}.}
Since $q$ is bounded and $\{\mu^{\ALP Y}_n\}_{n\in\Na}$ is a set of probability measures, the first summand on the right side of \eqref{eqn:qfnf} converges to zero as $n\rightarrow\infty$. This implies that for any sequence $\{x_n\}_{n\in\Na}$ converging to $x$, the sequence of measures $\{\nu^{x_n}_n\}_{n\in\Na}$ converges to $\nu^x$ in the weak-* topology. The second step of the proof is complete.

{\it Step 3:} Let $k_n(x):=g(x,\nu^x_n)$ and $k(x):=g(x,\nu^x)$. Using the result of Step 2, we know that for any sequence $\lbrace x_n\rbrace$ converging to $x$, we get $(x_n,\nu^{x_n}_n)\longrightarrow (x,\nu^x)$. This implies the desired property
\begin{align*}
   \int k_n(x) d\mu^{\ALP X}_n\longrightarrow \int k(x) d\mu^{\ALP X} \qquad\text{ as } n\rightarrow\infty
\end{align*}
by \cite[Theorem 5.5, p. 34]{billing1968} or \citet[Theorem 8.4.1 (iii), p. 195]{bogachev2006b}. The above equation is precisely \eqref{eqn:convI}, which completes the proof of the theorem.

\section{Proof of Lemma \ref{lem:cl1}}\label{app:cl1}
Let us define $\FLD D$ to be the set of all $A\in \FLD B_{\left[ 0,1\right[}$ for which \eqref{rademacher} holds. We show that $\FLD D$ form a Dynkin system. Clearly, $\left[ 0,1\right[$ satisfies (\ref{rademacher}), which implies $\left[ 0,1\right[\in\FLD D$. We show $\FLD D$ form a Dynkin system using two steps. 

{\it Step 1:} Let $\{A_m\}_{m\in\Na}\subset\FLD D$ be a mutually disjoint sequence of sets. We want to show that $\cup_{m=1}^{\infty} A_m\in\FLD D$. For every $n\in\Na$, we have
\begin{align*}
   \lambda\left( \left( \cup_{m=1}^\infty A_m\right) \bigcap F_n^{-1}(1)\right) =\sum_{m=1}^\infty \lambda\left( A_m\cap F_n^{-1}(1)\right) 
\end{align*}
Recall that since $A_m\in\FLD D$, $\frac{1}{2}\lambda(A_m) = \lf{n}\lambda(A_m\cap F_n^{-1}(1))$. Since $\lambda(A_m) = \lambda(A_m\cap F_n^{-1}(1))+\lambda(A_m\cap F_n^{-1}(0))$ for any $n\in\Na$, we get $\frac{1}{2}\lambda(A_m) = \lf{n}\lambda(A_m\cap F_n^{-1}(0))$. By taking the counting measure on natural numbers, we apply Fatou's lemma to get
\beqq{\frac{1}{2}\sum_{m=1}^\infty  \lambda(A_m)\leqslant \liminf\limits_{n\rightarrow \infty} \sum_{m=1}^\infty \lambda\left( A_m\cap F_n^{-1}(1)\right) \leqslant \limsup\limits_{n\rightarrow \infty} \sum_{m=1}^\infty \lambda\left( A_m\cap F_n^{-1}(1)\right),\\
\frac{1}{2} \sum_{m=1}^\infty\lambda(A_m)\leqslant \liminf\limits_{n\rightarrow \infty} \sum_{m=1}^\infty \lambda\left( A_m\cap F_n^{-1}(0)\right)\leqslant \limsup\limits_{n\rightarrow \infty} \sum_{m=1}^\infty \lambda\left( A_m\cap F_n^{-1}(0)\right).}
We also have $\sum_{m=1}^\infty \lambda\left( A_m\cap F_n^{-1}(1)\right)=\sum_{m=1}^\infty \lambda\left( A_m\right) -\sum_{m=1}^\infty \lambda\left( A_m\cap F_n^{-1}(0)\right)$, which further implies
\beqq{\limsup\limits_{n\rightarrow \infty}\sum_{m=1}^\infty \lambda\left( A_m\cap F_n^{-1}(1)\right)=\sum_{m=1}^\infty \lambda\left( A_m\right) -\liminf\limits_{n\rightarrow \infty}\sum_{m=1}^\infty \lambda\left( A_m\cap F_n^{-1}(0)\right)\leq\frac{1}{2}\sum_{m=1}^\infty  \lambda(A_m). }
Inequalities proved above imply $ \lambda\left( \left( \cup_{m=1}^\infty A_m\right) \bigcap F_n^{-1}(1)\right)\rightarrow \frac{1}{2}\lambda\left( \cup_{m=1}^{\infty} A_m\right)$ as $n\rightarrow\infty$, that is,
\beqq{\lambda_n\left( \cup_{m=1}^{\infty} A_m\right)\rightarrow \lambda\left( \cup_{m=1}^{\infty} A_m\right)\quad \text{ as }n\rightarrow\infty.}
Thus, $\cup_{m=1}^{\infty} A_m$ belongs to $\FLD D$.

{\it Step 2:} Note that $\frac{1}{2}=\lambda(F_n^{-1}(1))=\lambda(A\cap F_n^{-1}(1))+\lambda(A^\complement\cap F_n^{-1}(1))$. Since $\lambda(A\cap F_n^{-1}(1))\rightarrow \frac{1}{2}\lambda(A)$, we get $\lambda(A^\complement\cap F_n^{-1}(1))\rightarrow \frac{1}{2}\lambda(A^\complement)$. This completes the second step.

Steps 1 and 2, along with the fact that $\left[ 0,1\right[\in\FLD D$, proves that $\FLD D$ is a Dynkin system.

Now, we show that $\FLD D = \FLD B_{\left[ 0,1\right[} $. Since for $m\in \mathbb{N}$ and each $k\in \lbrace 1,...,2^{m}-1\rbrace$, the sets of the form $\left[ \frac{k}{2^m},\frac{k+1}{2^m}\right[$ belong to $\mathcal{D}$ and are intersection stable, the Dynkin system $\mathcal{D}$ is equal to the $\sigma$-algebra generated by $\left[ \frac{k}{2^m},\frac{k+1}{2^m}\right[,\: k\in \lbrace 1,...,2^{m}-1\rbrace,\: m\in\Na$, which is the Borel $\sigma$-algebra $\FLD B_{\left[ 0,1\right[}$. This shows that $\FLD D =  \FLD B_{\left[ 0,1\right[ }$. Thus, $\lambda_n$ converges to $\lambda$ setwise as $n\rightarrow\infty$.

\section{Proof of Theorem \ref{thm:settoi}}\label{app:settoi}
Let $g$ a bounded and uniformly continuous function on $\ALP X\times \wp_w(\ALP Y)$ (where $\ALP X\times\wp_w(\ALP Y)$ is endowed with the product metric). For any $\mu_n$ and $x\in\ALP X$, define $\nu^x_n\in\wp(\ALP Y)$ as $\nu^x_n(\cdot):=\mu_n(\cdot|x)$. Similarly, define $\nu^x\in\wp(\ALP Y)$ as $\nu^x(\cdot):=\mu(\cdot|x)$. By the definition of the convergence in the topology of information as weak-* convergence of the measure $\psi(\mu_n)$ to $\psi(\mu)$ and by Exercise 10 on p.203 in \citet{bauer2001}, it suffices to show that
\begin{align*}
  \left\vert \sum_{x\in \ALP X} g(x,\nu^{x}) \mu^{\ALP X}(x)-\sum_{x\in \ALP X} g(x,\nu^{x}_n) \mu^{\ALP X}_n(x)\right\vert\longrightarrow 0
\end{align*}
where $\mu^{\ALP X}$ denotes the marginal distribution on $\ALP X$. Towards this end, note that
\beq{& &\left\vert \sum_{x\in \ALP X} g(x,\nu^{x}) \mu^{\ALP X}(x)-\sum_{x\in \ALP X} g(x,\nu^{x}_n) \mu^{\ALP X}_n(x)\right\vert \nonumber\\
 &&  \leq \left\vert \sum_{x\in \ALP X} g(x,\nu^{x}) \mu^{\ALP X}(x)-\sum_{x\in \ALP X} g(x,\nu^{x}) \mu^{\ALP X}_n(x)\right\vert+ \sum_{x\in \ALP X} \Big\vert g(x,\nu^{x})-g(x,\nu^{x}_n) \Big\vert \mu^{\ALP X}_n(x).\label{eqn:setxy}}
We show that the two summands in the equation above converges to $0$ as $n\rightarrow\infty$ in two steps. 

{\it Step 1:} Note that $g(x,\nu^{x})$ is a bounded measurable function on $\ALP X$. Furthermore, $\mu^{\ALP X}_n$ converges setwise to $\mu^{\ALP X}$. By Exercise 11.2 in \citet{stokey1989} we then have that
\begin{align*}
   \left\vert \sum_{x\in \ALP X} g(x,\nu^{x}) \mu^{\ALP X}(x)-\sum_{x\in \ALP X} g(x,\nu^{x}) \mu^{\ALP X}_n(x)\right\vert \underset{n\rightarrow \infty}{\longrightarrow} 0.
\end{align*}
Thus, the first summand in \eqref{eqn:setxy} converges to $0$ as $n\rightarrow\infty$.

{\it Step 2:} The support of $\mu^{\ALP X}$, denote $S_{\mu}\subset\ALP X$ is given by
\beqq{ S_{\mu}:=\big\{x\in\ALP X:\mu^{\ALP X}(x)>0\big\}.}
Since $\mu^{\ALP X}$ is a probability measure on a discrete countable space, for each $\varepsilon>0$, there exists a finite set $F_{\varepsilon}\subseteq S_{\mu}$ such that $\mu^{\ALP X}(\ALP X\backslash F_{\varepsilon})<\varepsilon$. For any $x\in F_{\varepsilon}$, the regular conditional distribution $\mu(B\vert x)$ for any $B\in \FLD B(\ALP Y)$ is given by $\frac{\mu(B\cap \lbrace x\rbrace)}{\mu(\lbrace x\rbrace)}$. Since $\{\mu_n\}_{n\in\Na}$ converges setwise to $\mu$, for each $B\in \FLD B(\ALP Y)$, we get 
\begin{align*}
\nu^x_n(B) = \mu_n(B\vert x)=\frac{\mu_n(B\cap \lbrace x\rbrace)}{\mu_n(\lbrace x\rbrace)}\longrightarrow \frac{\mu(B\cap \lbrace x\rbrace)}{\mu(\lbrace x\rbrace)}=\mu(B\vert x) = \nu^x(B) \quad \text{ as } n\rightarrow\infty.
\end{align*}
So, $\{\nu^x_n\}_{n\in\Na}$ converges setwise to $\nu^x$ for every $x\in F_{\varepsilon}$, which further implies that $\{\nu^x_n\}_{n\in\Na}$ converges to $\nu^x$ in weak-* topology for every $x\in F_{\varepsilon}$. Next, we have
\beqq{\sum_{x\in \ALP X} \Big\vert g(x,\nu^{x})-g(x,\nu^{x}_n) \Big\vert \mu^{\ALP X}_n(x)=\sum_{x\in F_{\varepsilon}} \Big\vert g(x,\nu^{x})-g(x,\nu^{x}_n) \Big\vert \mu^{\ALP X}_n(x)+\sum_{x\in \ALP X\backslash F_{\varepsilon}} \Big\vert g(x,\nu^{x})-g(x,\nu^{x}_n) \Big\vert \mu^{\ALP X}_n(x).}
We now show that the first summand in the equation above converges to $0$ as $n\rightarrow\infty$. Fix $x\in F_{\varepsilon}$. Since $\nu^x_n\ws\nu^x$ and $g$ is uniformly continuous, for $\varepsilon>0$, there exists $N_x\in\Na$ such that for $n>N_x$, we have that $\left\vert g(x,\nu^{x})-g(x,\nu^{x}_n) \right\vert<\varepsilon$. Define $N_1:=\max_{x\in F_{\varepsilon}} N_x$, which is finite since  $F_{\varepsilon}$ is a finite set. Then, for any $n\geq N_1$, we have $\left\vert g(x,\nu^{x})-g(x,\nu^{x}_n) \right\vert<\varepsilon$ for all $x\in F_{\varepsilon}$.


Next, we prove the second summand in the equation above converges to $0$ as $n\rightarrow\infty$. Since $\mu^{\ALP X}_n$ converges setwise to $\mu^{\ALP X}$, There exists $N_2$ such that for any $n>N_2$, $\mu^{\ALP X}_n(\ALP X\backslash F_{\varepsilon})<2\varepsilon$. This implies
\beqq{\sum_{x\in \ALP X\backslash F_{\varepsilon}} \Big\vert g(x,\nu^{x})-g(x,\nu^{x}_n) \Big\vert \mu^{\ALP X}_n(x)<4\|g\|_{\infty} \varepsilon \quad \text{ for } n\geq N_2.}

Thus, we proved that for any $\varepsilon>0$, there exists $N:=\max \{N_1,N_2\}$ such that for any $n\geq N$, we have
\beqq{\sum_{x\in \ALP X} \Big\vert g(x,\nu^{x})-g(x,\nu^{x}_n) \Big\vert \mu^{\ALP X}_n(x)< (1+4\|g\|_{\infty})\varepsilon.}
This completes Step 2 of the proof.

Now, using the inequality in \eqref{eqn:setxy} and the results from Steps 1 and 2 above, we know that $\mu_n\rightarrow\mu$ in the topology of information, which completes the proof of the theorem.

\section{Jordan's Result as a corollary of Theorem \ref{thm:wstoi3}}\label{app:jordan}

\begin{remark}
\citet{hellwig1996} also considers the relation between the topology of information and the assumptions in \citet{jordan1977}. However, we have been unable to verify the remark on \citet[p. 449]{hellwig1996}, in which he claims the following: If $\{\mu_n\}_{n\in\Na}\subset\wp_w(\ALP X\times\ALP Y)$ is a weak-* convergent sequence of measures converging to $\mu$ such that the conditional measure $\mu_n(\cdot|x)$ is a continuous function from $\ALP X$ to $\wp_w(\ALP Y)$ for every $n\in\Na$, then $\mu_n\rightarrow\mu$ as $n\rightarrow\infty$ in the topology of  information. We provide a counterexample to this claim in Example \ref{exm:hellwigremark} below. {\hfill $\Box$}
\end{remark}


We now provide a counterexample to Hellwig's claim stated in the remark above, which adapts the example given on \citet[p. 1371]{jordan1977}.

\begin{example}\label{exm:hellwigremark}
Consider $\ALP X=\left\{ 2,2+\frac{1}{n},n\in \mathbb{N}\right\} $ and $\ALP Y=\left\{ 1,3\right\}$, where $\ALP X$ is endowed with the subspace topology of  the real line, and $\ALP Y$ is endowed with discrete topology. Define $\nu\in\wp(\ALP X\times\ALP Y)$ such that $\nu \left( \left\{ \left( 2,1\right) \right\} \right) =\nu \left( \left\{\left( 2,3\right) \right\} \right) =\frac{1}{2}$. Consider a sequence $\{\nu _{n}\}_{n\in\Na}\subset\wp(\ALP X\times\ALP Y)$ with $\nu_{n}\left( \left\{ \left( 2,1\right) \right\}\right) =\nu _{n}\left( \left\{ \left( 2+\frac{1}{n},3\right) \right\}\right) =\frac{1}{2}$. Therefore, $\nu_{n}$ converges to $\nu$ as $n\rightarrow\infty$ in the weak-* topology. For each $n$, the conditional distribution on $\ALP Y$ given $x$ is given by 
\beqq{\nu_{n}\left(1|2\right) =1, \quad \nu_{n}\left( 3\left| 2+\frac{1}{k}\right. \right) = 1 \text{ for } k=n.}
For $k\in\Na\setminus\{n\}$, let us assume that the conditional measures are given by
\beqq{\nu _{n}\left( 1\left| 2+\frac{1}{k}\right. \right) =1 \text{ for }k \neq n,\: k\in\Na.}
For fixed $n\in\Na$, the conditional distribution $\nu_{n}(\cdot|x)$ is a continuous function on $\ALP X$. Nevertheless, we have
\begin{equation*}
\psi\left( \nu _{n}\right) =\left\{ 
\begin{array}{cl}
\left( 2,\ind{1}\right) & \text{with probability }\frac{1}{2} \\ 
\left( 2+\frac{1}{n},\ind{3}\right) & \text{with probability }\frac{1}{2}
\end{array}
\right. ,\quad
\psi\left( \nu \right) =\ind{\left(2,\frac{1}{2}\ind{1}+\frac{1}{2}\ind{3}\right)}.
\end{equation*}
Now, consider a bounded continuous function $h:\wp_w(\ALP Y)\rightarrow \mathbb{R}$ such that $h(\ind{1})=h(\ind{3})=0$ and $h(\frac{1}{2}\ind{1}+\frac{1}{2}\ind{3})=1$\footnote{Since $\wp_w(\ALP Y)$ is a metric space, Urysohn's Lemma (\citet{willard2004}) implies that such a function exists.}. We have $\int_{\ALP X\times\wp_w(\ALP Y)} h(\zeta)d\psi(\nu_n)=0$ for all $n\in\Na$ and $\int_{\ALP X\times\wp_w(\ALP Y)} h(\zeta)d\psi(\nu)=1$. Thus, $\psi\left(\nu _{n}\right) $ does not converge to $\psi\left( \nu \right)$ in the weak-* topology. Hence, $\nu _{n}$ does not converge to $\nu $ in
the topology of information.{\hfill$\Box$}
\end{example}

To relate the topology of information to the assumptions of \citet{jordan1977}, one can instead argue as follows:
\citet{jordan1977} assumes that there is a first countable space $H$ and that there are continuous functions $\lambda : H\rightarrow \wp_w(\ALP X\times \ALP Y)$ and $\nu :H\times \ALP X\rightarrow \wp_w(\ALP Y)$\footnote{Note that this is more demanding than just to require that for each fixed $\eta \in H$, $\nu\left( .,\eta \right) :\ALP X\rightarrow \wp_{w}\left( \ALP Y\right) $ has a continuous version and that $\nu \left( \eta _{n}\right) $ converges to $\nu \left( \eta \right) $ in weak-* topology whenever $\eta _{n}\rightarrow \eta $ in $H$.} such that for each $\eta \in H$, $\nu(\eta,x)$ is the regular conditional distribution of $\lambda(\eta)$ given $x\in \ALP X$. Let $\{\eta_n\}_{n\in\Na}\subset H$ be a convergent sequence with limit $\eta$. In the proof of the Theorem \ref{thm:wstoi3}, define the functions $h_n:\ALP X\rightarrow\ALP X\times\wp_w(\ALP Y)$ and $h:\ALP X\rightarrow\ALP X\times\wp_w(\ALP Y)$ as $h_n(x)=(x,\nu(\eta_n,x))$ and $h(x)=(x,\nu(\eta,x))$ for $n\in\Na$. The continuity 
properties of the map $\nu$ imply that $h_n(x_n)\rightarrow h(x)$ in $\ALP X\times \wp_w(\ALP Y)$ whenever $x_n\rightarrow x$, which holds for all $x\in\ALP X$. Convergence of the sequence in the topology of information follows using the same steps as in the proof of Theorem \ref{thm:wstoi3}.

\section*{Acknowledgements}
This paper is in parts based on the working paper ``A note on the Hellwig maximum theorem for sequential decisions under uncertainty'' of the first author. The first author wants to thank Martin Hellwig for his extensive and extremly helpful comments, in particular for proposing him in a private communication the statement of the continuity result in Section \ref{sec:main}. The second author is thankful to Prof. Serdar Y\"{u}ksel and Prof. Tamer Ba\c{s}ar for several useful discussions about the topology of information on the space of measures. The authors are also thankful to the editor and the reviewer for their insightful comments. The work of second author was supported in part by the AFOSR MURI Grant FA9550-10-1-0573. 
 

\end{document}